\documentclass[11pt,reqno]{amsart}
\usepackage{mathrsfs}
\usepackage{amsmath,stackengine}
\usepackage{url}
\usepackage{mathtools}
\usepackage{latexsym,epsfig,amssymb,amsmath,amsthm,color,url,bm}
\usepackage[inline,shortlabels]{enumitem}
\usepackage{hyperref}
\usepackage[foot]{amsaddr}
\usepackage{amsmath,amsbsy}
\usepackage{mwe}
\RequirePackage[numbers]{natbib}
\usepackage{mathptmx}
\usepackage[text={16cm,24cm}]{geometry}

\allowdisplaybreaks 
 \numberwithin{equation}{section}
\newtheorem{theorem}{Theorem}[]

\newtheorem{lemma}[theorem]{Lemma}

\newtheorem{conjecture}[theorem]{Conjecture}

\newcommand\groupequation[2][17pt]{%
  \setbox0=\hbox{$\displaystyle#2$}%
  \stackengine{0pt}{\copy0}{%
    \makebox[\linewidth]{\hfill$\left.\rule{0pt}{\ht0}\right\}$\kern#1}}
    {O}{c}{F}{T}{L}
}

\newcommand\Item[1][]{%
  \ifx\relax#1\relax  \item \else \item[#1] \fi
  \abovedisplayskip=0pt\abovedisplayshortskip=0pt~\vspace*{-\baselineskip}}

\theoremstyle{definition}

\theoremstyle{definition}

\theoremstyle{definition}

\title{Antimagicness of Tensor product for some wheel related graphs with star}
\date{22/11/2023}
\author{Vinothkumar Latchoumanane$^{ \ 1}$ and Murugan Varadhan$^{ \ 2,*}$}
\address{${}^{1,2}$Department \ of \ Mathematics, \ School \ of \ Advanced \ Sciences, \  Vellore \ Institute \ of\  Technology, \  Vellore \ - \ $632014$,\ India.} 
\address{* Correspondence} 
\email{${}^{1}$noblevino01@gmail.com, ${}^{2,*}$murugan.v@vit.ac.in}
\begin{document}
\bibliographystyle{plainnat}
\begin{abstract}
A graph $G$ with $p$ vertices and $q$ edges has an antimagic labelling if there is a bijection from the graph's edge set to the label set $\left\{1,2, \cdots, q \right\}$ such that $p$ vertices must have distinct vertex sums, where the vertex sums are determined by adding up all the edge labels incident to each vertex $v$ in $V(G)$. Hartsfield and Ringel \cite{Ringel1} in the book "Pearls in Graph Theory" conjectured that every connected graph is antimagic, with the exception of $P_2$. In this study, we identified a class of connected graphs that lend credence to the conjecture. In this article, we proved that the tensor product of a wheel and a star, a helm and a star, and a flower and a star is antimagic.
\end{abstract}
\subjclass[2020]{}Primary 05C78, Secondary 05C76
\maketitle
\section{Introduction}

\par In this paper, we considered only simple, undirected, finite, and connected graphs. For a given graph $G, \left|V(G)\right| = p$ and $\left|E(G)\right|=q.$ A graph $G$ is antimagic if there exists a bijective function $f : E(G) \rightarrow \left\{1,2, \cdots, q\right\}$ such that $\phi_{f} (u) \neq \phi_{f} (v) $ for every $u,v \in V(G),$ where $\phi_{f} (u)$ is defined as the sum of labels assigned to the edges which are incident to the vertex $u.$ In the introductory book "Pearls in Graph theory", Hartsfield and Ringel \cite{Ringel1} introduced antimagic graph. Moreover, they proposed the following conjecture.
\begin{conjecture} \cite{Ringel1}
    All connected graphs other than $P_2$ is antimagic.
\end{conjecture}
For the general results on the antimagicness of connected graphs, Alon et al., \cite{Alon} proved for dense graphs. Also, they further investigated for the graph with minimum degree $\delta (G) \geq c \ log \   p$ for some constant value $'c'$ or with maximum degree $\Delta(G) \geq p-2$ is antimagic. Further, yilma \cite{Yilma} extended the graph with maximum degree $\Delta(G) \geq p-3$. 

\par To support the conjecture $1.1$, several authors constructed connected graphs using product of graphs. More particularly, Yu Chang Liang  and Xuding Zhu \cite{LiangTCS} restricted graph $G_2$ with the following conditions: (i) if $\left|V(G_2)\right|-1 \leq \left|E(G_2)\right|$ and (ii) every connected component of the graph $G_2$ having a odd degree vertex (or) atleast $G_2$ must have $2\left|V(G_2)-2\right|$ edges, thus the graph $G_2 \times K_2$ is antimagic.
Yongxi cheng \cite{ycheng1} proved that cartesian product of two (or) more components of regular graphs admits antimagicness and also they investigated for cartesian product of cycle with cycle graph. Tao ming wang \cite{WangDM} shown that the cartesian product of $P_m (m \geq 2)$ and $P_n (n \geq 2),$ $P_m (m \geq 2)$ and $G$, $C_m (m \geq 3)$ and $G$, where $G$ is a $l$ - regular graph for $l \geq 1$ are antimagic. 

\par Tao ming wang also proved that the lexicographic product of $G \times H$ is antimagic where $G$ is any arbitrary graph and $H$ is a $t-$ regular graph for $t \geq 1.$ Yingyu Lu \cite{YingyuMCS} shown that the lexicographic product of $K_{m,n}$ and the $P_k$ is antimagic. Wenhui Ma et al., \cite{WenhuiAKCE} proved lexicographic product of path and path is antimagic. Antimagic labeling for join graphs was done by Martin Ba{\v{c}}a et al. \cite{Baca}. For more results on product of graphs, the readers can refer to Gallian survey \cite{Gallian}. In this paper, we proved that the tensor product of wheel and a star, a helm and a star, and a flower and a star is antimagic.
 \section{ Preliminaries}
The tensor product of graph $G$ and $H$ is defined as the vertex sets $V(G) \times V(H),$ the vertex pairs $(x_1,y_1)$ and $(x_2,y_2)$ is adjacent in $G \times H$ iff $x_1 x_2 \in E(G)$ and $y_1 y_2 \in E(H).$ The tensor product of graph satisfies the associative and commutative property.   P.M. Weichsel in the year $(1963)$  proved that "the tensor product of graph $G_1 \times G_2$ is disconnected iff any one of the graph $G_1$ or $G_2$ contains no odd cycle" \cite{Weichsel}. In this paper, we considered graphs consisting of odd cycles, i.e., Wheel graph, Helm graph and Flower graph. 

\par The star graph is a complete bipartite graph $K_{1,n}$, the vertex $u_{0}$ is adjacent to all the vertices from the set $\left\{u_{1}, u_{2}, \cdots , u_{n}\right\}$, where $u_{0}$ is a central vertex. The vertex $u_{i}, 1 \leq i \leq m$ forms a path $P_{m}$, by joining the vertex $u_{i}$ and $u_{i+1}.$ The cycle graph $C_{m}$, is formed from path graph by joining the vertex $u_{m}$ and $u_{1}$. The Wheel graph $W_{m}$, is obtained from the cycle graph by adding a central vertex $u_{0}$ and joining it to each vertex of $u_{i}, 1 \leq i \leq m.$ The Helm graph $H_{m}$, is obtained from the wheel graph by taking $m$  vertices namely $u_{m+i}, 1 \leq i \leq m$ and join an edge between $u_{i}$ to $u_{m+i}$ for $1 \leq i \leq m.$ The flower graph $Fl_{m}$, can be formed from the helm graph by joining the vertex $u_{m+i}, 1 \leq i \leq m$ to the vertex $u_{0}$. In section $3.1$ we discussed the antimagicness for the tensor product of Wheel and Star, the antimagicness for the tensor product of Helm and Star were discussed in section $3.2$ and in section $3.3$ we discussed the antimagicness for the tensor product of Flower and Star.
\section{Main Results}
In this section, we constructed irregular connected graph with the help of tensor product of some wheel related graphs.
\subsection{Antimagicness for tensor product of wheel and star} 
By definition of tensor product, the vertex set of $W_m \times K_{1,n}$ are defined as,
\begin{equation} \nonumber
V(W_m \times K_{1,n}) = \left\{(u_i, v_j) = w_{i}^{j}: 0 \leq i \leq m,  0 \leq j \leq n\right\}
\end{equation}
We can write this vertex set of $G$ as,
\begin{equation} \nonumber
V(G) = \left\{w_{0}^{0}\right\} \cup \alpha^{'} \cup  \alpha^{''} \cup  \alpha^{'''}
\end{equation}
where, $\alpha^{'} = \left\{w_{i}^{j}: 1 \leq i \leq m, 1 \leq j \leq n\right\},
\alpha^{''} = \left\{w_{i}^{0}: 1 \leq i \leq m\right\}$ and \\
$\alpha^{'''} = \left\{w_{0}^{j}: 1 \leq j \leq n\right\}$. And the edge set of $G$ as,
\begin{equation} \nonumber
E(G) = \beta^{'} \cup \beta^{''} \cup \beta^{'''}
\end{equation}
where,
\begin{equation} \nonumber
\begin{split}
\beta^{'} &= \left\{(w_{0}^{0}, w_{i}^{j}): 1 \leq i \leq m, 1 \leq j \leq n\right\} \\
\beta^{''} &= \left\{(w_{i}^{j}, w_{i+1}^{0}): 1 \leq i \leq m-1 \& 1 \leq j \leq n\right\}  \cup \left\{(w_{i}^{0}: w_{i+1}^{j}), 1 \leq i \leq m-1, 1 \leq j \leq n\right\} \\& \cup \left\{(w_{m}^{0}, w_{1}^{j}): 1 \leq j \leq n\right\} \cup \left\{(w_{1}^{0}, w_{m}^{j}): 1 \leq j \leq n\right\} \\
\beta^{'''} &= \left\{(w_{i}^{0}, w_{0}^{j}): 1 \leq i \leq m, 1 \leq j \leq n\right\}
\end{split}
\end{equation}
Note that  $\left|E(G)\right|=4mn$.  

\begin{theorem} \label{Wheelmodd}
    The tensor product of Wheel $W_m, m \geq 3$ and star $K_{1,n}$ where $m$ is odd and $n \geq 1$  admits an antimagic labeling.
\end{theorem}
\begin{proof} The labeling of the edges of graph $G$ are given in the following steps. \\
\emph{Step 1: } Labeling the edges of $(w_{0}^{0},w_{i}^{j}) \in \beta^{'}$, for each $i, 1 \leq i \leq m$ and $j, 1 \leq j \leq n$,\\
When $i \neq 1$,
\begin{equation}  \nonumber
f (w_{0}^{0}, w_{i}^{j}) =
    \begin{cases}
       	2mn+(i-1)n+(2j-1)  &,  if \ \text{i is odd}  \\
				3mn+(i-1)n+(2j-1)  &,  if \ \text{i is even} 
    \end{cases}
\end{equation}
When $i=1$, 
\begin{equation}  \nonumber
f (w_{0}^{0}, w_{i}^{j}) =
    \begin{cases}
       	2mn+(2j-1) &,  if \ \text{n is odd }   \\
				2mn+2j  &,  if \ \text{n is even } 
    \end{cases}
\end{equation}
\emph{Step 2: } Labeling the edges of $\beta^{''}$ for each $i, 1 \leq i \leq m$ and $j, 1 \leq j \leq n$,
\begin{align*}
f (w_{m}^{0}, w_{1}^{j})  & = j  \\
f (w_{m}^{j}, w_{1}^{0}) & = mn+j
\end{align*}
For each $i, 1 \leq i \leq m-1$,
\begin{equation}   \nonumber
f (w_{i}^{j}, w_{i+1}^{0}) =
    \begin{cases}
        in+j &, if \ \text{i is even} \\
        mn+in+j &, if \ \text{i is odd}
    \end{cases}
\end{equation}
For each $i, 1 \leq i \leq m-1$,
\begin{equation} \nonumber
f (w_{i}^{0}, w_{i+1}^{j}) =
    \begin{cases}
        in+j &, if \ \text{i is even} \\
        mn+in+j &, if \ \text{i is odd}
    \end{cases}
\end{equation}
\emph{Step 3: } Labeling the edges of $\beta^{'''}$ for each $i, 1 \leq i \leq m$ and $j, 1 \leq j \leq n$, \\
When $i \neq 2$,
\begin{equation}  \nonumber
f (w_{i}^{0}, w_{0}^{j})=
    \begin{cases}
        2mn+(i-2)n+2j &, if \ \text{i is even}   \\
       	3mn+in+2j &, if \ \text{i is odd and} \ i \neq m\\
				 2mn+(i-1)n+2j &, if \ \text{i = m}
    \end{cases}
\end{equation}
When $i=2$, 
\begin{equation}  \nonumber
f (w_{i}^{0}, w_{0}^{j})=
    \begin{cases}
        2mn+2j &, if \ \text{n is odd}  \\
       	2mn+2j-1 &, if \ \text{n is even} 
    \end{cases}
\end{equation}
From the above labeling schemes, we obtained the following vertex sums for every vertex $v \in V(G),$\\
The vertex sum of vertex $(w_{0}^{0})$ is,
\begin{equation}  \nonumber
\phi_f (w_{0}^{0}) = 
    \begin{cases}
        3m^2 n^2 &, \ \text{if n is odd} \\
				3m^2 n^2 + n &, \ \text{if n is even}
    \end{cases}
\end{equation}
For each $i, 1 \leq i \leq m$ and $1 \leq j \leq n$, the vertex sum of the vertex, $w_{i}^{j} \in \alpha^{'}$ is,\\
When $i \neq 1$,
\begin{equation}  \nonumber
\phi_{f} (w_{i}^{j}) =
    \begin{cases}
       (2m+3i-2)n+4j-1  &,  if \ \text{i is odd } \\
				(5m+3i-2)n+4j-1  &,  if \ \text{i is even } 
    \end{cases}
\end{equation}
When $i = 1$,
\begin{equation}  \nonumber
\phi_{f} (w_{i}^{j}) =
    \begin{cases}
       (2m+3i-2)n+4j-1  &,  if \ \text{n is odd }  \\
			 (2m+3i-2)n+4j  &,  if \ \text{n is even } 
    \end{cases}
\end{equation}
For each $i, 1 \leq i \leq m$ and the vertex sum of the vertex, $w_{i}^{0} \in \alpha^{''}$ is,\\
When $i \neq 2$
\begin{equation}  \nonumber
\phi_f (w_{i}^{0}) = 
    \begin{cases}
        (5m+3i+1)n^2 + 2n &, if \ \text{i is odd}, i \neq m \\
       	(2m+3i-1)n^2 + 2n  &,  if \ \text{i is even} \\
				5in^2 +2n &, if \ \text{i = m}
    \end{cases}
\end{equation}
When $i = 2$
\begin{equation}  \nonumber
\phi_f (w_{i}^{0}) = 
    \begin{cases}
       	(2m+3i-1)n^2 + 2n  &,  if \ \text{n is odd} \\
				(2m+3i-1)n^2 + n  &,  if \ \text{n is even} 
    \end{cases}
\end{equation}
Finally, the vertex sum of the vertex, $w_{0}^{j} \in \alpha^{'''}$ for each $j, \ 1 \leq j \leq n$ is calculated as,
\begin{equation}  \nonumber
\phi_f (w_{0}^{j}) =
    \begin{cases}
       	(3mn+2j-n)m  &,  if \ \text{n is odd} \\
				(3mn+2j-n)m - 1  &,  if \ \text{n is even} 
    \end{cases}
\end{equation}
From the above vertex sums, we have the following observations,\\
(1) From the vertex of $w_{i}^{j} \in \alpha^{'}$ for each $i, 1 \leq i \leq m$. We have, 
\begin{equation} \nonumber
\phi_{f} (w_{1}^{j}) < \phi_{f} (w_{3}^{j}) < \cdots <  \phi_{f} (w_{m}^{j}) <  \phi_{f} (w_{2}^{j})  < \phi_{f} (w_{4}^{j}) <  
\cdots < \phi_{f} (w_{m-1}^{j})  
\end{equation}
further, for each $i, 1 \leq i \leq m$ there exist $j_{1}, j_{2}, 1 \leq j_{1}, j_{2} \leq n$ such that if $j_{1} < j_{2}$ we have $\phi_f (w_{i}^{j_{1}}) < \phi_f (w_{i}^{j_{2}})$.\\
(2) From the vertex of $w_{i}^{0} \in \alpha^{''}$ for each $i, 1 \leq i \leq m$. We have,
\begin{equation} \nonumber
\phi_{f} (w_{2}^{0}) < \phi_{f} (w_{4}^{0}) < \cdots < \phi_{f} (w_{m-1}^{0}) < \phi_{f} (w_{m}^{0}) < \phi_{f} (w_{1}^{0}) < \phi_{f} (w_{3}^{0})  < \cdots < \phi_{f} (w_{m-2}^{0})
\end{equation}
(3) From the vertex of $w_{0}^{j} \in \alpha^{'''}$ for each $j, 1 \leq j \leq n$. We have,
\begin{equation} \nonumber
\phi_f (w_{0}^{1}) < \phi_f (w_{0}^{2}) < \cdots < \phi_f (w_{0}^{n})
\end{equation}
Therefore, it is clear that no two vertices belongs to the same set of vertices, $(\alpha^{'}, \alpha^{''}, \alpha^{'''})$ are equal. And from the above observations and the defined vertex sums, we have the following:
\begin{itemize}
	\item The vertex sum of the vertex $w_{0}^{0}$ is greater than all the other vertex sums.
	\item If the vertex having a odd sum then we have, \\
	When $n$ is odd, $\phi_{f} (w_{3}^{0}) < \phi_{f} (w_{0}^{0}) $.\\
	When $n$ is even, $\phi_{f} (w_{m-1}^{n}) < \phi_{f} (w_{0}^{1}) $.
\item If the vertex having a even sum then we have, \\
	When $n$ is odd, $\phi_{f} (w_{m-1}^{n}) < \phi_{f} (w_{2}^{0}) $.\\
	When $n$ is even, $\phi_{f} (w_{1}^{1}) < \phi_{f} (w_{2}^{0}) $.
\end{itemize}
Hence, for any two distinct vertices in $G$ receives distinct vertex sum. Therefore, $G$ is antimagic.
\end{proof}
\begin{theorem} \label{Wheelmeven}
    The tensor product of Wheel $W_m, m \geq 3$ and star $K_{1,n}$ where $m$ is even and $n \geq 1$  admits an antimagic labeling.
\end{theorem}
\begin{proof}
The labeling of the edges of graph $G$ are given in the following steps. \\
\emph{Step 1: } Labeling the edges of $(w_{0}^{0},w_{i}^{j}) \in \alpha^{'}$ for each $i, 1 \leq i \leq m$ and $j, 1 \leq j \leq n$,\\
When $i \neq 1$,
\begin{equation} \nonumber
f (w_{0}^{0}, w_{i}^{j}) =
    \begin{cases}
       	(2m+i-1)n+2j-1  &,  if \ \text{i is odd for} \ j, 1 \leq j \leq n  \\
	(4m+1-i)n-5+2j  &,  if \ \text{i is even for} \ j, 1 \leq j \leq n 
    \end{cases}
\end{equation}
When $i = 1$,
\begin{equation} \nonumber
f (w_{0}^{0}, w_{i}^{j}) =
    \begin{cases}
       	2mn+(2j-1)  &,  if \ \text{n is odd} \\
	2mn+2j  &,  if \ \text{n is even} 
    \end{cases}
\end{equation}
\emph{Step 2: } Labeling the edges of $\beta^{''}$ for each $i, 1 \leq i \leq m$ and $j, 1 \leq j \leq n$,
\begin{align*}
f \left(w_{m}^{0},w_{1}^{j}\right) &= j \\
f \left(w_{1}^{0},w_{m}^{j}\right) &= mn+j
\end{align*}
For each $i,$ $1 \leq i \leq m-1$ and $1 \leq j \leq n$
\begin{equation} \nonumber
f (w_{i}^{j}, w_{i+1}^{0}) =
    \begin{cases}
       	in+j  &,  if \ \text{i is odd } \\
				n(2m-i)+j  &,  if \ \text{i is even}
    \end{cases}
\end{equation}
For each $i,$ $1 \leq i \leq m-1$ and $1 \leq j \leq n$
\begin{equation} \nonumber
f (w_{i}^{0}, w_{i+1}^{j}) =
    \begin{cases}
		n(2m-i)+j  &,  if \ \text{i is odd} \\
       	in+j  &, if \ \text{i is even } 
    \end{cases}
\end{equation}
\emph{Step 3: } Labeling the edges of $\beta^{'''}$ for each $i, 1 \leq i \leq m$ and $j, 1 \leq j \leq n$, \\
When $i = 2$,\\
\begin{equation} \nonumber
f (w_{i}^{0}, w_{0}^{j}) = 
    \begin{cases}
        2mn+2j &, if \ \text{n is odd} \\
       	2mn+2j-1  &, if \ \text{n is even}
    \end{cases}
\end{equation}
 When $i \neq 2$,
\begin{equation}  \nonumber
f (w_{i}^{0}, w_{0}^{j}) = 
    \begin{cases}
        n(2m+i-2)+2j &, if \ 4 \leq i \leq  2\left\lfloor \dfrac{m}{4}\right\rfloor \ \& \ i \ even\\
       	n(2m+2\left\lfloor \dfrac{m}{4}\right\rfloor) +2j &, if \ i = m \\
				n(2m+i)+2j &, if \ 2\left\lfloor \dfrac{m}{4}\right\rfloor + 2 \leq i \leq m-2  \ \& \ i  \ even \\
				n(4m-1-i)+2j &, if \ 2 \left\lceil \dfrac{m}{4}\right\rceil+1 \leq i \leq m-1  \ \& \ i \ odd\\
				n(4m-2\left\lceil \dfrac{m}{4}\right\rceil) + 2j &, if \ i = 1  \\
				n(4m+1-i)+2j &, if \ 3 \leq i \leq 2\left\lceil \dfrac{m}{4}\right\rceil -1  \ \& \ i \ odd \\
    \end{cases}
\end{equation}
From the above labeling schemes, we obtained the following vertex sums for every vertex $v \in V(G).$ The vertex sum of vertex $(w_{0}^{0})$ is,
\begin{equation} \nonumber
\phi_{f} (w_{0}^{0}) = 
    \begin{cases}
        3m^2 n^2 &, if \ n \ is \  odd\\
       	3m^2 n^2 + n &, if \ n \ is \ even
    \end{cases}
\end{equation}
For each $i, 1 \leq i \leq m$ and $1 \leq j \leq n$, the vertex sum of the vertex $w_{i}^{j} \in \alpha^{'}$ is,\\
When $i \neq 1$,
\begin{equation} \nonumber
\phi_{f} (w_{i}^{j}) = 
    \begin{cases}
  (2m+3i-2)n+4j-1 &, if \ i \ is \ odd\\
  (8m-3i+2)n+4j-5 &, if \ i  \ is  \ even
    \end{cases}
\end{equation}
When $i = 1$,
\begin{equation} \nonumber
\phi_{f} (w_{i}^{j}) = 
    \begin{cases}
   (2m+3i-2)n+4j-1 &, if \ \text{n is odd} \\
    (2m+3i-2)n+4j &, if \ \text{n is even}
    \end{cases}
\end{equation}
For each $i, 1 \leq i \leq m$ and the vertex sum of the vertex $w_{i}^{0} \in \alpha^{''}$ is,\\
When $i \neq 2$,
\begin{equation} \nonumber
\phi_{f} (w_{i}^{0}) = 
    \begin{cases}
        n^2 (3i+2m-1)+2n &, if \ 4 \leq i \leq  2\left\lfloor \dfrac{m}{4}\right\rfloor \ for \ i \ even \\
       	n^2 (3m+1+2\left\lfloor \dfrac{m}{4}\right\rfloor)+2n &, if \ i = m \\
				n^2 (3i+2m+1)+2n &, if \ 2\left\lfloor \dfrac{m}{4}\right\rfloor + 2 \leq i \leq m-2 \ for \ i \ even \\
				n^2 (8m-3i+2)+2n &, if \ m-1  \leq i \leq 2 \left\lceil \dfrac{m}{4}\right\rceil-1 \ for \ i \ odd \\
				n^2 (7m+1-2\left\lceil \dfrac{m}{4}\right\rceil)+2n &, if \ i = 1 \\
				n^2 (8m-3i+4)+2n &, if \ 3 \leq i \leq 2\left\lceil \dfrac{m}{4}\right\rceil -1 \ for \ i \ odd
    \end{cases}
\end{equation}
When $i = 2$,
\begin{equation} \nonumber
\phi_{f} (w_{i}^{0}) = 
    \begin{cases}
        n^2 (2m+5)+2n &, if \ \text{n is odd} \\
  n^2 (2m+5)+n &, if \ \text{n is even} \\
    \end{cases}
\end{equation}
Finally, the vertex sum of the vertex $w_{0}^{j} \in \alpha^{'''}$ for each $j, \ 1 \leq j \leq n$ is calculated as,
\begin{equation} \nonumber
\phi_{f} (w_{0}^{j}) = 
    \begin{cases}
        (3mn+2j-n)m &, if \ \text{n is odd} \\
  (3mn+2j-n)m-1 &, if \ \text{n is even}
    \end{cases}
\end{equation}
From the above vertex sums, we have the following observations,\\
(1) From the vertex of $w_{i}^{j} \in \alpha^{'}$. We have, 
\begin{equation} \nonumber
\phi_{f} (w_{1}^{j}) <  \phi_{f} (w_{3}^{j})  < \cdots < \phi_{f} (w_{m-1}^{j}) < \phi_{f} (w_{m}^{j}) < \phi_{f} (w_{m-2}^{j})  < \cdots < \phi_{f} (w_{4}^{j}) < \phi_{f} (w_{2}^{j})  
\end{equation}
further, for each $i, 1 \leq i \leq m$ there exist $j_{1}, j_{2}, 1 \leq j_{1}, j_{2} \leq n$ such that if $j_{1} < j_{2}$ we have $\phi_f (w_{i}^{j_{1}}) < \phi_f (w_{i}^{j_{2}})$.\\
(2) From the vertex of $w_{i}^{0} \in \alpha^{''}$ for each $i, 1 \leq i \leq m$. We have,
\begin{multline} \nonumber
\phi_f (w_{2}^{0}) < \phi_f (w_{4}^{0}) < \cdots < \phi_f (w_{2\left\lfloor \dfrac{m}{4}\right\rfloor}^{0}) < \phi_f (w_{m}^{0}) < \phi_f (w_{2\left\lfloor \dfrac{m}{4}\right\rfloor + 2}^{0}) <  \cdots <\\ \phi_f (w_{m-2}^{0}) < \phi_f (w_{m-1}^{0})  < \cdots < \phi_f (w_{2\left\lceil \dfrac{m}{4}\right\rceil+1}^{0})< \phi_f (w_{1}^{0}) < \phi_f (w_{2\left\lceil \dfrac{m}{4}\right\rceil}^{0})\\ < \cdots < \phi_f (w_{3}^{0})
\end{multline}
(3) From the vertex of $w_{0}^{j} \in \alpha^{'''}$ for each $j, 1 \leq j \leq n$. We have,
\begin{equation}  \nonumber
\phi_f (w_{0}^{1}) < \phi_f (w_{0}^{2}) < \cdots < \phi_f (w_{0}^{n})
\end{equation}
Therefore, it is clear that no two vertices belongs to the same set of vertices are equal. And from the above observations and also from the defined vertex sums, we have the following:
\begin{itemize}
	\item The vertex sum of the vertex $w_{0}^{0}$ is greater than all the other vertex sums.
	\item If the vertex having a odd sum then we have, \\
	When $n$ is odd, $\phi_{f} (w_{3}^{0}) < \phi_{f} (w_{0}^{0}) $.\\
	When $n$ is even, $\phi_{f} (w_{2}^{n}) < \phi_{f} (w_{0}^{1}) $.
\item If the vertex having a even sum then we have, \\
	When $n$ is odd, $\phi_{f} (w_{3}^{n}) < \phi_{f} (w_{0}^{1}) $.\\
	When $n$ is even, $\phi_{f} (w_{1}^{n}) < \phi_{f} (w_{2}^{0}) $.
\end{itemize}
Hence, for any two distinct vertices we obtained distinct vertex sum. Therefore, $G$ is antimagic.
\end{proof}
From theorem \ref{Wheelmodd} and \ref{Wheelmeven}, we obtain the following theorem.
\begin{theorem} \label{Wheelmain}
    The tensor product of Wheel $W_m, m \geq 3$ and star $K_{1,n}, n \geq 1$  admits an antimagic labeling.
\end{theorem}
\subsection{Antimagicness for tensor product of helm and star}
\noindent Let the vertex set of the tensor product of $H_{m} \times K_{1,n} = G$ is,
\begin{equation} \nonumber
V(G) = \left\{(u_i, v_j) = w_{i}^{j}: 0 \leq i \leq 2m,  0 \leq j \leq n\right\}
\end{equation}
Note that, the tensor product of wheel and star is the induced subgraph of the tensor product of helm and star. So, we can write the vertex set of $G$ can be written using the vertex set of tensor product of wheel and star, 
\begin{equation} \nonumber
V(G) = \left\{w_{0}^{0}\right\} \cup \gamma^{'} \cup  \gamma^{''} \cup  \gamma^{'''}
\end{equation}
where, $ \gamma^{'} = \alpha^{'} \cup \left\{w_{i}^{j}: m+1 \leq i \leq 2m, 1\leq j \leq n\right\},
\gamma^{''} = \alpha^{''} \cup \left\{w_{m+i}^{0}: 1 \leq i \leq m \right\}$ and 
$\gamma^{'''} = \alpha^{'''}$
where $\alpha^{'}, \alpha^{''}$ and $\alpha^{'''}$ are defined in section $3.1$.\\
The edge set of $G$ is defined as,
\begin{align*} 
E(G) = \delta^{'} \cup \delta^{''} \cup \delta^{'''}
\end{align*}
where, $\delta^{'} = \beta^{'}, \delta^{'''}  = \beta^{'''} $ and $ \delta^{''}=\beta^{''} \cup \left\{(w_{i}^{j}, w_{m+i}^{0}): 1 \leq i \leq m, 1 \leq j \leq n\right\}
\cup \\ \left\{(w_{m+i}^{j}, w_{i}^{0}): 1 \leq i \leq m, 1 \leq j \leq n\right\}$. \\
where $\beta^{'}$ and $\beta^{'''}$ are defined in section $3.1.$ Note that $\left|E(G)\right| = 6mn.$
Before, proving the main theorem we have the following lemma.
\begin{lemma} \label{Helmn=1}
The graph $G= H_{m} \times K_{1,1}$ is antimagic for $m \geq 3.$
\end{lemma}
\begin{proof}
We label the edges of the graph $G$ by using the function, $f_{1}: E(G) \rightarrow \left\{1,2, \cdots, 6m\right\}$. \\
The edges of $\delta^{'}$ is labelled as,
\begin{align*} \nonumber
f_{1} (w_{0}^{0}, w_{i}^{1})=& 2m+2i\ for \ i= 1,2, \cdots, m
\end{align*}
The edges of $\delta^{''}$ is labelled as,
\begin{align*} \nonumber
f_{1} (w_{i}^{1}, w_{i+1}^{0})=&
     \begin{cases}
        2m+1 &, if \ i = 1   \\
        2m+4i-1 &, if \ 2 \leq i \leq m-1, 
    \end{cases}\\
		f_{1} (w_{1}^{1}, w_{m}^{0})=& 2m+3, \\
		f_{1} (w_{m}^{1}, w_{1}^{0})=& 3(2m-1), \\
		f_{1} (w_{i}^{0}, w_{i+1}^{1})=&
     \begin{cases}
        2m+4i+1 &, if \ 1 \leq i \leq m-2   \\
        6m-1 &, if \ i = m-1, 
    \end{cases}\\
		f_{1} (w_{i}^{1}, w_{m+i}^{0})=& 2i-1, \ \text{for} \ 1 \leq i \leq  m \\
			f_{1} (w_{i}^{0}, w_{m+i}^{1})=&
     \begin{cases}
		m+2 &, if \ i = 1 \ \& \ m \ \text{is even}  \\
				m+3 &, if \ i = 1 \ \& \ m \ \text{is odd}  \\
		 2(i-1) &, if \ 2 \leq i \leq \left\lceil \dfrac{m-1}{2}\right\rceil   \\
		m+1 &, if \ i = \dfrac{m+1}{2} \ \& \ m \ \text{is odd}  \\
		2(i+1) &, if \ \left\lfloor \dfrac{m+3}{2}\right\rfloor \leq i \leq m-1  \\
       m-1 &, if \ i = m \ \& \ m \ \text{is odd}  \\
				m &, if \ i = m \ \& \ m \ \text{is even} 
				\end{cases}
\end{align*}
The edges of $\delta^{'''}$ is labelled as,
\begin{align*}
f_{1} (w_{i}^{0}, w_{0}^{1})=
\begin{cases}
5m+2 &, if \ i = 1 \ \& \ m \ \text{is even}  \\
				5m+3 &, if \ i = 1 \ \& \ m \ \text{is odd}  \\
4m+2(i-1) &, if \ 2 \leq i \leq \left\lceil \dfrac{m-1}{2}\right\rceil   \\
5m+1 &, if \ i = \dfrac{m+1}{2} \ \& \ m \ \text{is odd}  \\
2i+4m+2 &, if \ \left\lfloor \dfrac{m+3}{2}\right\rfloor \leq i \leq m-1  \\
        5m-1 &, if \ i = m \ \& \ m \ \text{is odd}  \\
				5m &, if \ i = m \ \& \ m \ \text{is even}  			
    \end{cases}
\end{align*}
The edge labels induces the vertex sums of $G$ as follows,
\begin{align*}
\phi_{f_{1}} (w_{0}^{0}) =& 3m^{2} + m, \\
\phi_{f_{1}} (w_{i}^{1}) =& 6m+12i-5 \ \text{for} \ 1 \leq i \leq m, \\
\phi_{f_{1}} (w_{m+i}^{1}) =& f(w_{i}^{0},w_{m+i}^{1}),\\ 
\phi_{f_{1}} (w_{i}^{0} ) =&
\begin{cases}
14m+8 &,  if \ i= 1 \ \& \ m \ \text{is odd} \\ 
				14m+6 &,  if \ i= 1 \ \& \ m  \ \text{is even} \\
8m+14 &, if \ i=2   \\
8m+12i-8 &, if \ 3 \leq i \leq \left\lceil \dfrac{m-1}{2}\right\rceil  \\
9m+8i-3 &, if \ i = \dfrac{m+1}{2} \  \& m \ \text{is odd} \\
        8m+12i &, for \left\lfloor \dfrac{m+3}{2}\right\rfloor \leq i \leq m-2 \\
					10(2m-1) &, if \ i=m-1 \\
				14m-4 &, if \ i= m \ \& \ m  \ \text{is odd} \\ 
				14m-2 &, if \ i= m \ \& \ m  \ \text{is even} 		
    \end{cases}\\
		\phi_{f_{1}} (w_{m+i}^{0} ) =& f(w_{i}^{1}, w_{m+i}^{0})\\
 \phi_{f_{1}} (w_{0}^{1} ) =&	 5m^{2} + m	
\end{align*}
Thus, we have the following observation:\\
(1). For each $i, 1 \leq i \leq m$, we have 	$\phi_{f_{1}} (w_{1}^{1}) < \phi_{f_{1}} (w_{2}^{1}) < \cdots < \phi_{f_{1}} (w_{m}^{1}) $ and
	\begin{itemize}
		\item If $m$ is odd, \\
	$\phi_{f_{1}} (w_{2}^{0}) < \phi_{f_{1}} (w_{3}^{0}) < \cdots < \phi_{f_{1}} (w_{\left\lceil \dfrac{m-1}{2}\right\rceil}^{0} ) < \phi_{f_{1}} (w_{m}^{0}) < \phi_{f_{1}}  (w_{\dfrac{m+1}{2}}^{0}) < \phi_{f_{1}} (w_{1}^{0}) < \phi_{f_{1}} (w_{\left\lfloor \dfrac{m+3}{2}\right\rfloor}^{0}) < \phi_{f_{1}} (w_{\left\lfloor \dfrac{m+5}{2}\right\rfloor}^{0})  <  \cdots < \phi_{f_{1}} (w_{m-1}^{0})$.
	\item 	If $m$ is even, \\
	$\phi_{f_{1}} (w_{2}^{0} ) < \phi_{f_{1}} (w_{3}^{0}) < \cdots < \phi_{f_{1}} (w_{\left\lceil \dfrac{m-1}{2}\right\rceil}^{0} ) < \phi_{f_{1}} (w_{m}^{0}) < \phi_{f_{1}} ( w_{1}^{0}) < \phi_{f_{1}} (w_{\left\lfloor \dfrac{m+3}{2}\right\rfloor}^{0})  < \phi_{f_{1}} (w_{\left\lfloor \dfrac{m+5}{2}\right\rfloor}^{0})  < \cdots < \phi_{f_{1}} (w_{m-1}^{0})$.
	\end{itemize}
Clearly, the above vertex sums are distinct due to the following reasoning:
\begin{enumerate}
	\item The one degree vertices is lesser from all the other vertices so it is different from every other vertices and it is different between themselves.
	\item The vertex $w_{i}^{1}$ for $i = 1,2, \cdots, m$ is having a odd sum so it is different from all the other vertices, since the remaining vertices are having a even sums. 
	\item For the remaining vertices we have,\\
	If $m=3, \phi_{f_{1}} (w_{0}^{0}) < \phi_{f_{1}} (w_{3}^{0}) < \phi_{f_{1}} (w_{2}^{0}) < \phi_{f_{1}} (w_{0}^{1}) < \phi_{f_{1}} (w_{1}^{0}) $.\\
	If $m=4, \phi_{f_{1}} (w_{2 }^{0}) < \phi_{f_{1}} (w_{0}^{0}) < \phi_{f_{1}} (w_{m}^{0}) < \phi_{f_{1}} (w_{1}^{0}) < \phi_{f_{1}} (w_{3}^{0}) < \phi_{f_{1}} (w_{0}^{1}) $.\\
		If $m=5, \phi_{f_{1}} (w_{2}^{0}) < \phi_{f_{1}} (w_{5}^{0}) < \phi_{f_{1}} (w_{3}^{0}) < \phi_{f_{1}} (w_{1}^{0}) < \phi_{f_{1}} (w_{0}^{0}) < \phi_{f_{1}} (w_{4}^{0}) < \phi_{f_{1}} (w_{0}^{1}) $.\\
			If $m \geq 7, \phi_{f_{1}} (w_{m-1}^{0}) < \phi_{f_{1}} (w_{0}^{0}) < \phi_{f_{1}} (w_{0}^{1}) $.
\end{enumerate}
Thus, any two distinct vertices of $G= H_{m} \times K_{1,1}$ getting distinct vertex sums. Therefore, $G$ is antimagic.
\end{proof}
Next, we prove the antimagicness of $H_{m} \times K_{1,n}$ for the values $m \geq 3, n \geq 2$ which proved in theorem  \ref{Helmmodd} and theorem \ref{Helmmeven}.
\begin{theorem} \label{Helmmodd}
The tensor product of Helm $W_m, m \geq 3$ and star $K_{1,n}$ where $m$ is odd and $n \geq 2$  admits antimagic labeling.
\end{theorem}
\begin{theorem}\label{Helmmeven}
The tensor product of helm $H_m, m \geq 4$ and star $K_{1,n}$ where $m$ is even and $n \geq 2$  admits antimagic labeling.
\end{theorem}
Before proving theorem \ref{Helmmodd} and theorem \ref{Helmmeven}, we classify the tensor product of Helm $H_{m}, m \geq 3$ and star $K_{1,n}, n \geq 2$ into three cases based on the values of $m$ and $n$.
\begin{equation}   \nonumber
G \cong
     \begin{cases}
        G^{'} &, if \ m \geq n \ \& \ n \ \text{is odd}   \\
				G^{''} &, if \ m < n \ \& \ n \ \text{is odd}   \\
        G^{'''} &, if  \ n \ \text{is even}, 
    \end{cases}
\end{equation}
In the proof of theorem \ref{Helmmodd} and \ref{Helmmeven}, we use the labelling of $G^{'}, G^{''}$ and $G^{'''}$ as $g^{'}, g^{''}$ and $g^{'''}$ respectively.\\
\textbf{Proof of theorem \ref{Helmmodd}: } The labeling of the edges of graph $G$ are given in the following steps. \\
\emph{Step 1: } Labeling the edges of $(w_{0}^{0},w_{i}^{j}) \in \delta{'}$ for $1 \leq i \leq m$ and $1 \leq j \leq n$ is defined as follows:
\begin{equation}   \nonumber
\begin{split}
g^{'} (w_{0}^{0},w_{i}^{j}) &= 
    \begin{cases}
        4mn+(i-1)n+2j &, if \ \text{i is odd}  \\
				5mn+(i-1)n+2j &, if \ \text{i is even},
    \end{cases}\\
g^{''} (w_{0}^{0},w_{i}^{j}) &= g^{'} (w_{0}^{0},w_{i}^{j})   \\
g^{'''} (w_{0}^{0},w_{i}^{j}) &=
    \begin{cases}
        4mn+2j &, if \ i = 1  \\
				 g^{'} (w_{0}^{0},w_{i}^{j}) - 1 &, if \  i \neq 1 
    \end{cases}
		\end{split}
\end{equation}
\emph{Step 2: } Labeling the edges of $\delta^{''}$ as follows:\\
First, label the edges of $(w_{i}^{j}, w_{m+i}^{0})$ for each $i, 1 \leq i \leq m$ and $1 \leq j \leq n$,\\
\begin{equation}   \nonumber
\begin{split}
g^{'} (w_{i}^{j}, w_{m+i}^{0}) &= 
     \begin{cases}
        (i-1)n+2j-1 &, if \ \text{i is odd}   \\
        (m-1)n+in+2j-1 &, if \ \text{i is even}, 
    \end{cases}\\
g^{''} (w_{i}^{j}, w_{m+i}^{0}) &= 4mn +g^{'} (w_{i}^{j}, w_{m+i}^{0}) \\
\noindent \text{When} \ n = 2, \\
g^{'''} (w_{i}^{j}, w_{m+i}^{0}) &= 
     \begin{cases}
        2j &, if \ i = 1  \ \& \ m \geq 3 \\
				g^{'} (w_{i}^{j}, w_{m+i}^{0}) + 1 &, if \ i \neq 1  \ \& \ m \geq 3
    \end{cases}\\
\noindent \text{When} \ n \geq 3, \\
g^{'''} (w_{i}^{j}, w_{m+i}^{0}) &= 
     \begin{cases}
		4mn+2j-1 &, if \ i = 1  \ \& \ m = 3 \\
        2j-1 &, if \ i = 1  \ \& \ m \geq 5 \\
				4mn+ g^{'} (w_{i}^{j}, w_{m+i}^{0}) + 1 &, if \ i \neq 1 \ \& \ m = 3 \\
			  g^{'} (w_{i}^{j}, w_{m+i}^{0}) + 1 &, if \ i \neq 1 \ \& \ m \geq 5 				
    \end{cases}
		\end{split}
\end{equation}
Next, label the edges of $(w_{m+i}^{j}, w_{i}^{0})$ for each $i, 1 \leq i \leq m$ and $1 \leq j \leq n$,
\begin{equation}   \nonumber
\begin{split}
g^{'} (w_{m+i}^{j}, w_{i}^{0}) &=  
     \begin{cases}
        n(m+i)+2j &, if \ \text{i is odd} \ \& \ i \neq m    \\
        (i-2)n+2j &, if \ \text{i is even}    \\
				(m-1)n+2j &, if \ i = m 
    \end{cases}\\
g^{''} (w_{m+i}^{j}, w_{i}^{0}) &= g^{'} (w_{m+i}^{j}, w_{i}^{0})		\\
\text{When} \ n = 2, \\
g^{'''} (w_{m+i}^{j}, w_{i}^{0}) &= 
     \begin{cases}
        2j-1 &, if \ i = 2    \\
        g^{'} (w_{m+i}^{j}, w_{i}^{0}) -1 &, if \ i \neq 2
    \end{cases} \\
\text{When} \ n \geq 3, \\
g^{'''} (w_{m+i}^{j}, w_{i}^{0}) &= 
     \begin{cases}
        2j &, if \ i = 2    \\
        g^{'} (w_{m+i}^{j}, w_{i}^{0}) -1 &, if \ i \neq 2
    \end{cases}
\end{split}
\end{equation}
Finally, we label the remaining edges of $\delta^{''}$ for each $i, 1 \leq i \leq m$ and for $j, 1 \leq j \leq n,$
\begin{equation}   \nonumber
\begin{split}
g^{'} (w_{i}^{0}, w_{i+1}^{j}) &= 
    \begin{cases}
        3mn+in+j &, if \ \text{i is odd}  \\
				2mn+in+j &, if \ \text{i is even}
    \end{cases} \\
		g^{''} (w_{i}^{0}, w_{i+1}^{j}) &= g^{'} (w_{i}^{0}, w_{i+1}^{j}) \\
		g^{'''} (w_{i}^{0}, w_{i+1}^{j}) &= g^{'} (w_{i}^{0}, w_{i+1}^{j}) \\
g^{'} (w_{i}^{j}, w_{i+1}^{0}) &= 
    \begin{cases}
        3mn+in+j &, if \ \text{i is even}  \\
				2mn+in+j &, if \ \text{i is odd}
    \end{cases}
    \end{split}
    \end{equation}

  \begin{equation}   \nonumber
\begin{split}  
g^{''} (w_{i}^{j}, w_{i+1}^{0}) &= g^{'} (w_{i}^{j}, w_{i+1}^{0}) \\
g^{'''} (w_{i}^{j}, w_{i+1}^{0}) &= g^{'} (w_{i}^{j}, w_{i+1}^{0}) \\
g^{'} (w_{1}^{j}, w_{m}^{0}) &=  2mn+j \\
g^{''} (w_{1}^{j}, w_{m}^{0}) &= g^{'} (w_{1}^{j}, w_{m}^{0}) \\
g^{'''} (w_{1}^{j}, w_{m}^{0}) &= g^{'} (w_{1}^{j}, w_{m}^{0}) \\
g^{'} (w_{1}^{0}, w_{m}^{j}) &= 3mn+j \\
 g^{''} (w_{1}^{0}, w_{m}^{j}) &= g^{'} (w_{1}^{0}, w_{m}^{j}) \\
 g^{'''} (w_{1}^{0}, w_{m}^{j}) &= g^{'} (w_{1}^{0}, w_{m}^{j}) 
\end{split}
\end{equation}
\emph{Step 3: }  We label the edges of $(w_{i}^{0}, w_{0}^{j}) \in \delta^{'''}$ for each $i, 1 \leq i \leq m$ and $1 \leq j \leq n, $
\begin{equation}   \nonumber
\begin{split}
g^{'} (w_{i}^{0}, w_{0}^{j}) &= 
    \begin{cases}
        4mn+(i-2)n+2j-1 &, if \ \text{i is even}  \\
				5mn-n +2j -1 &, if \ i = m \\
				5mn+in+2j-1 &, if \ \text{i is odd} \ \& \ i \neq m \\
    \end{cases} 
		\end{split}
\end{equation}	
			
\begin{equation}   \nonumber
\begin{split}	
g^{''} (w_{i}^{0}, w_{0}^{j}) &=  g^{'} (w_{i}^{0}, w_{0}^{j}) - 4mn \\
g^{'''} (w_{i}^{0}, w_{0}^{j}) &= 
    \begin{cases}
		    2j-1 &, if \ i = 2  \ \& \ m = 3 \\
				4mn+2j-1 &, if \ i = 2  \ \& \ m \geq 5 \\
				g^{'} (w_{i}^{0}, w_{0}^{j}) + 1 - 4mn &, if \ i \neq  2  \ \& \ m = 3 \\
				g^{'} (w_{i}^{0}, w_{0}^{j}) + 1 &, if \ i \neq 2 \ \& \ m \geq 5
    \end{cases}
		\end{split}
\end{equation}
From the above labeling schemes, we obtained the following vertex sums for every vertex $v \in V(G)$. \\
The vertex sum of vertex $(w_{0}^{0})$ is,
\begin{equation} \nonumber
\begin{split}
\phi_{g^{'}} (w_{0}^{0}) & = \phi_{g^{''}} (w_{0}^{0})  =   5m^{2} n^{2} + mn,\\
\phi_{g^{'''}} (w_{0}^{0}) & =   5m^{2} n^{2} + n,
\end{split}
\end{equation}
For each $i, 1 \leq i \leq m$ and $1 \leq j \leq n, $ the vertex sum of the vertex $w_{i}^{j} \in \gamma^{'}$ is,
\begin{equation}  \nonumber
\begin{split}
\phi_{{g}^{'}} (w_{i}^{j}) &= 
    \begin{cases}
        8mn+6j+4in-3n-1 &, if \ \text{i is odd}  \\
				12mn+4in-3n+6j-1 &, if \ \text{i is even},
    \end{cases}\\
		\phi_{{g}^{''}} (w_{i}^{j}) &= \phi_{{g}^{'}} (w_{i}^{j}) +4mn \\
			\text{When} \ n = 2, \\
		\phi_{{g}^{'''}} (w_{i}^{j}) &= 
    \begin{cases}
		   \phi_{{g}^{'}} (w_{i}^{j}) +1 &, if \ i = 1 \\
        \phi_{{g}^{'}} (w_{i}^{j})  &, if \ i \neq 1 
    \end{cases}\\
\end{split}		
	\end{equation} 	
\begin{equation}  \nonumber
\begin{split}		
			\text{When} \ n \geq 3, \\
		\phi_{{g}^{'''}} (w_{i}^{j}) &= 
    \begin{cases}
		    4mn+ \phi_{{g}^{'}} (w_{i}^{j}) &, if \ m = 3 \\
        8mn+6j+n-1 &, if \ i = 1 \ \& \ m \geq 5 \\
				\phi_{{g}^{'}} (w_{i}^{j}) &, if \ i \neq 1 \ \& \ m \geq 5 \\
    \end{cases}
\end{split}
\end{equation}
The vertex sum of the vertex $w_{m+i}^{j}$ gets their vertex sums as same as of the edge label $(w_{m+i}^{j},w_{i}^{0})$ under their labelled function.  \\
The vertex sum of the vertex $w_{i}^{0} \in \gamma^{''}$ for $1 \leq i \leq m$ is,
\begin{equation}   \nonumber
\begin{split}
\phi_{{g^{'}}} (w_{i}^{0}) &= 
    \begin{cases}
        12mn^{2} + 4in^{2} + 2n^{2} + 2n &, if \ \text{i is odd}  \ \& \ i \neq m \\
				8mn^{2} + 4in^{2} -2n^{2} + 2n &, if \ \text{i is even} \\
				12mn^{2} + 2n &, if \ i = m
    \end{cases} \\
\phi_{{g^{''}}} (w_{i}^{0}) &=  \phi_{{g^{'}}} (w_{i}^{0}) - 4mn^{2} \\
\text{When} \ n = 2, \\
\phi_{{g^{'''}}} (w_{i}^{0}) &= 
    \begin{cases}
		    \phi_{{g^{'}}} (w_{i}^{0}) - n &, if \ i = 2 \\
        \phi_{{g^{'}}} (w_{i}^{0}) &, if \ i \neq 2
    \end{cases} \\
\text{When} \ n \geq 3,\\
\phi_{{g^{'''}}} (w_{i}^{0}) &= 
    \begin{cases}
		    \phi_{{g^{'}}} (w_{i}^{0}) &, if \ m=3 \\
        8mn^{2} + 6n^{2} + 3n &, if \ i=2 \ \& \ m \geq 5 \\
			  \phi_{{g^{'}}} (w_{i}^{0}) &, if \ i \neq 2 \ \& \ m \geq 5 
    \end{cases}
		\end{split}
\end{equation}
and the vertex sum of the vertex $w_{m+i}^{0} \in \gamma^{''} $ for each $i, 1 \leq i \leq m$ and $j, 1 \leq j \leq n,$
\begin{equation}  \nonumber
\begin{split}
\phi_{{g^{'}}} (w_{m+i}^{0}) &= 
    \begin{cases}
        in^{2} &, if \ \text{i is odd}  \\
				(m+i)n^{2} &, if \ \text{i is even},
    \end{cases} \\
\phi_{{g^{''}}} (w_{m+i}^{0}) &= 4mn^{2}  + \phi_{{g^{'}}} (w_{m+i}^{0})
\end{split}
\end{equation}
\begin{equation}  \nonumber
\begin{split}
\text{When} \ n = 2, \\
\phi_{{g^{'''}}} (w_{m+i}^{0})  &= 
    \begin{cases}
        n(n+1) &, \ if \ i = 1    \\
				\phi_{{g^{'}}} (w_{m+i}^{0})  &, \ if \ i \neq 1
    \end{cases} \\
		\text{When} \ n \geq 3, \\
\phi_{{g^{'''}}} (w_{m+i}^{0})  &= 
    \begin{cases}
        n^{2} &, \ if \ i = 1   \ \& \ m \geq 5  \\
				\phi_{{g^{'}}} (w_{m+i}^{0})  &, \ if \ i \neq 1   \ \& \ m \geq 5  \\
				4mn^{2} + n^{2} &, \ if \ i = 1   \ \& \ m = 3  \\
				4mn^{2} +\phi_{{g^{'}}} (w_{m+i}^{0}) + n &, \ if \ i \neq 1   \ \& \ m = 3  \\
    \end{cases}
		\end{split}
\end{equation}
Finally, the vertex sum of the vertex $w_{0}^{j} \in \gamma^{'''}$ for $j, 1 \leq j \leq n$ is,
\begin{equation} \ \nonumber
\begin{split}
\phi_{g^{'}} (w_{0}^{j}) &= 5m^{2}n-mn+(2j-1)m\\
\phi_{g^{''}} (w_{0}^{j})  &= \phi_{g^{'}} (w_{0}^{j}) - 4m^{2}n\\
\phi_{g^{'''}} (w_{0}^{j}) &=  
    \begin{cases}
        \phi_{{g^{'}}} (w_{0}^{j}) - 4mn^{2} + m -1 &, \ if \ m = 3  \\
				\phi_{{g^{'}}} (w_{0}^{j}) + m -1  &, \ if \ m \geq  5
    \end{cases}
\end{split}
\end{equation}
The observation $(1),(2)$ and $(3)$ as given in theorem.  (\ref{Wheelmodd}) is also true under the function $g^{'}, g^{''}$ and $g^{'''}$. Now, we also have the following observation.\\
(1) From the vertex $w_{m+i}^{j}$ for each $i, 1 \leq i \leq m$. We have,
\begin{multline} \nonumber
\phi_{{g}^{'}} (w_{m+2}^{j}) < \phi_{{g}^{'}} (w_{m+4}^{j}) < \cdots < \phi_{{g}^{'}} (w_{2m-1}^{j}) < \phi_{{g}^{'}} (w_{2m}^{j}) < \phi_{{g}^{'}} (w_{m+1}^{j}) <  \phi_{{g}^{'}} (w_{m+3}^{j}) < \cdots < \phi_{{g}^{'}} (w_{2m-2}^{j})
\end{multline}
further, for each $i, 1 \leq i \leq m$ there exist $j_{1}, j_{2}, 1 \leq j_{1}, j_{2} \leq n$ such that if $j_{1} < j_{2}$ we have $\phi_{{g}^{'}} (w_{m+i}^{j_{1}}) < \phi_{{g}^{'}} (w_{m+i}^{j_{2}})$.\\
(2) From the vertex $w_{i}^{0}$ for each $i, 1 \leq i \leq m$. We have,
\begin{multline}  \nonumber
\phi_{{g}^{'}} (w_{m+1}^{0}) < \phi_{{g}^{'}} (w_{m+3}^{0}) < \cdots < \phi_{{g}^{'}} (w_{2m}^{0}) < \phi_{{g}^{'}} (w_{m+2}^{0}) < \phi_{{g}^{'}} (w_{m+4}^{0}) < \cdots <  \phi_{{g}^{'}} (w_{2m-3}^{0}) < \cdots  < \phi_{{g}^{'}} (w_{2m-1}^{0}) 
\end{multline}
The same observations holds true for ${g}^{''}$ and ${g}^{'''}$. Therefore, it is clear that no two vertices belongs to the same set of vertices are equal. And from the above vertex observations and the defined vertex sums, we have the following:
\begin{itemize}
	\item The vertex sum of the vertex $w_{0}^{0}$ is greater than all the other vertex sums.
	\item If the vertex having a odd sum then we have, \\
$\phi_{{g}^{'}} (w_{2m-1}^{0})   < \phi_{{g}^{'}} (w_{0}^{1}), \phi_{{g}^{''}} (w_{0}^{n}) < \phi_{{g}^{''}} (w_{m+1}^{0}). $ \\
When $m \geq 5$ and $n \neq 2$, $\phi_{{g}^{'''}} (w_{2m-2}^{n}) < \phi_{{g}^{'''}} (w_{m-1}^{n}) < \phi_{{g}^{'''}} (w_{0}^{1})$. \\
When $m \geq 3$ and $n = 2$, $\phi_{{g}^{'''}} (w_{m-1}^{n}) < \phi_{{g}^{'''}} (w_{0}^{1})$.\\
When $m = 3, \phi_{{g}^{'''}} (w_{1}^{1}) > \phi_{{g}^{'''}} (w_{0}^{n}). $
\item If the vertex having a even sum then we have, \\
$\phi_{{g}^{'}} (w_{m+3}^{n}) < \phi_{{g}^{'}} (w_{1}^{1}) < \phi_{{g}^{'}} (w_{m-1}^{n}) < \phi_{{g}^{'}} (w_{2}^{0}), \phi_{{g}^{''}} (w_{2m-2}^{n}) < \phi_{{g}^{''}} (w_{1}^{1}) < \phi_{{g}^{''}} (w_{m-1}^{n}) < \phi_{{g}^{''}} (w_{2}^{0})$.\\
When $m \geq 5$ and $n \neq 2$, $\phi_{{g}^{'''}} (w_{m+2}^{n}) < \phi_{{g}^{'''}} (w_{m+1}^{0}) < \phi_{{g}^{'''}} (w_{2m-1}^{0}) < \phi_{{g}^{'''}} (w_{2}^{0})$.\\
When $m \geq 3$ and $n = 2$, $\phi_{{g}^{'''}} (w_{1}^{n}) < \phi_{{g}^{'''}} (w_{2}^{0})$.\\
When $m = 3, \phi_{{g}^{'''}} (w_{2}^{0}) > \phi_{{g}^{'''}} (w_{5}^{0}) > \phi_{{g}^{'''}} (w_{5}^{n})$.
\end{itemize} 
Hence, for any two distinct vertices in $G$, we obtained distinct vertex sum. Therefore, $G$ is antimagic.\\
\textbf{Proof of theorem \ref{Helmmeven}.} The labeling of edges of graph $G$ are given in the following steps. \\
\emph{Step 1: } Labeling the edges of $(w_{0}^{0}, w_{i}^{j}) \in \delta^{'}$ for each $i, 1 \leq i \leq m$ and for $j, 1 \leq j \leq n$ is defined as follows:
\begin{equation}   \nonumber
\begin{split}
g^{'} (w_{0}^{0},w_{i}^{j}) &=
    \begin{cases}
        4mn+(i-1)n+2j &, if \ \text{i is odd}  \\
				5mn+(m-i)n+2j &, if \ \text{i is even},
    \end{cases} \\
		 g^{''} (w_{0}^{0},w_{i}^{j}) &=  g^{'} (w_{0}^{0},w_{i}^{j}) \\
		g^{'''} (w_{0}^{0},w_{i}^{j}) &= 
    \begin{cases}
        4mn+2j &, if \ i = 1  \\
				g^{'} (w_{0}^{0},w_{i}^{j}) -1 &, if \ i \neq 1
    \end{cases}
		\end{split}
\end{equation}
\emph{Step 2: } Labeling the edges of $\delta^{''}$ as follows:\\
First, label the edges of $(w_{i}^{j}, w_{m+i}^{0})$ for each $i, 1 \leq i \leq m$ and $1 \leq j \leq n$,\\
\begin{equation}   \nonumber
\begin{split}
g^{'} (w_{i}^{j}, w_{m+i}^{0}) &= 
     \begin{cases}
        2j+n(i-1) &, if \ \text{i is odd}   \\
        2j+n(2m-i) &, if \ \text{i is even}, 
    \end{cases}\\
g^{''} (w_{i}^{j}, w_{m+i}^{0}) &= 4mn-1+g^{'} (w_{i}^{j}, w_{m+i}^{0}) \\
g^{'''} (w_{i}^{j}, w_{m+i}^{0}) &= 
     \begin{cases}
        2j-1 &, if \ i = 1  \text{and} \ n \neq 2  \\
        2j &, if \ i = 1  \text{and} \ n = 2 \\ 
				g^{'} (w_{i}^{j}, w_{m+i}^{0}) &, if \ i \neq 1
    \end{cases}
	\end{split}
\end{equation}
Next, label the edges of $(w_{m+i}^{j}, w_{i}^{0})$ for each $i, 1 \leq i \leq m$ and $1 \leq j \leq n$,
 \begin{equation} \nonumber
\begin{split}
g^{'} (w_{m+i}^{j}, w_{i}^{0}) &= 
    \begin{cases}
        2j-1+n(i-2) &, if \ 2 \leq i \leq  2\left\lfloor \dfrac{m}{4}\right\rfloor \ \& \ i \ even \\
        2n\left\lfloor \dfrac{m}{4}\right\rfloor + 2j-1 &, if  \  i =m, \\
				ni+2j-1 &, if \ 2\left\lfloor \dfrac{m}{4}\right\rfloor + 2 \leq i \leq m-2 \ \& \ \ i \ even \\
				n(2m-1-i)+2j-1 &, if \ 2 \left\lceil \dfrac{m}{4}\right\rceil + 1 \leq i \leq m-1 \ \& \ \ i \ odd \\
				2mn-2n\left\lceil \dfrac{m}{4}\right\rceil + 2j -1 &, if \ i=1 for  \ m \neq 4\\
				3mn-4n\left\lceil \dfrac{m}{4}\right\rceil + (3-i)n+2j-1 &, if \  3 \leq  i \leq 2\left\lceil \dfrac{m}{4}\right\rceil -1 \ \& \ \ i \ odd \ for \ m \neq 4\\
				4n+2j-1 &, if  \  i = 1 \ for \ m = 4\\
				6n+2j-1 &, if  \  i=3 \ for \ m=4
    \end{cases}\\	
g^{''} (w_{m+i}^{j}, w_{i}^{0}) &=
    \begin{cases}
		 2j - 1 &, if  \  i =2 \\
		g^{'} (w_{m+i}^{j}, w_{i}^{0}) + 1 &, if  \  i \neq 2
    \end{cases}\\
g^{'''} (w_{m+i}^{j}, w_{i}^{0}) &=
    \begin{cases}
		 2j &, if \ i =2 \ \text{and} \ n \neq 2 \\
		 2j-1 &, if \ i =2 \ \text{and} \ n = 2 \\
		g^{'} (w_{m+i}^{j},w_{i}^{0}) &,  if  \  i \neq 2
    \end{cases}
		\end{split}
\end{equation}
Finally, we label the remaining edges of $\delta^{''}$ for each $i, 1 \leq i \leq m$ and for $j, 1 \leq j \leq n,$
\begin{equation} \nonumber
\begin{split}
g^{'} (w_{i}^{j}, w_{m}^{0}) &= 2mn+j \\
g^{''} (w_{i}^{j}, w_{m}^{0}) &= g^{'} (w_{i}^{j}, w_{m}^{0}) \\
g^{'''} (w_{i}^{j}, w_{m}^{0}) &= g^{'} (w_{i}^{j}, w_{m}^{0}) \\
g^{'} (w_{m}^{j}, w_{1}^{0}) &= 3mn+j \\
g^{''} (w_{m}^{j}, w_{1}^{0}) &= g^{'} (w_{m}^{j}, w_{1}^{0}) \\
g^{'''} (w_{m}^{j}, w_{1}^{0}) &= g^{'} (w_{m}^{j}, w_{1}^{0})
\end{split}
\end{equation}
For each $i, 1 \leq i \leq m-1$ and $1 \leq j \leq n,$
\begin{equation}   \nonumber
\begin{split}
g^{'} (w_{i}^{j}, w_{i+1}^{0}) &= 
    \begin{cases}
        (2m+i)n+j &, if \ \text{i is odd}  \\
				(4m-i)n+j &, if \ \text{i is even}
    \end{cases} \\
g^{''} (w_{i}^{j}, w_{i+1}^{0}) &= g^{'} (w_{i}^{j}, w_{i+1}^{0})\\
g^{'''} (w_{i}^{j}, w_{i+1}^{0}) &= g^{'} (w_{i}^{j}, w_{i+1}^{0}) \\
g^{'} (w_{i}^{0}, w_{i+1}^{j}) &= 
    \begin{cases}
        (4m-i)n+j &, if \ \text{i is odd}  \\
				(2m+i)n+j &, if \ \text{i is even}
    \end{cases}\\
		g^{''} (w_{i}^{0}, w_{i+1}^{j}) &= g^{'} (w_{i}^{0}, w_{i+1}^{j}) \\
		g^{'''} (w_{i}^{0}, w_{i+1}^{j}) &= g^{'} (w_{i}^{0}, w_{i+1}^{j}) 
		\end{split}
\end{equation}
\emph{Step 3: }  We label the edges of $(w_{i}^{0}, w_{0}^{j}) \in \delta^{'''}$ for each $i, 1 \leq i \leq m$ and $1 \leq j \leq n, $
\begin{equation} \nonumber
\begin{split}
g^{'} (w_{i}^{0}, w_{0}^{j}) &= 
    \begin{cases}
        4mn+(i-2)n+2j-1 &, if \ 2 \leq i \leq  2\left\lfloor \dfrac{m}{4}\right\rfloor \ \& \ \ i \ even \\
        4mn+2\left\lfloor \dfrac{m}{4}\right\rfloor n +2j-1 &, if  \ i =m, \\
				4mn+ni+2j-1 &, if \ 2\left\lfloor \dfrac{m}{4}\right\rfloor + 2 \leq i \leq m-2 \ \& \ \ i \ even \\
				6mn-n(i+1)+2j-1 &, if \ 2 \left\lceil \dfrac{m}{4}\right\rceil + 1 \leq i \leq m-1 \ \& \ \ i \ odd \\
				6mn-2\left\lceil \dfrac{m}{4}\right\rceil n +2j-1 &, if \ i=1 \ for \ m \neq 4 \\
				4mn+4n+2j-1 &, if \ i=1 \ for \ m = 4 \\
				6mn+n-1+2j-ni &, if \  3 \leq  i \leq 2\left\lceil \dfrac{m}{4}\right\rceil -1 \ \& \ \ i \ odd \ and \ m \neq 4\\
				4mn+6n+2j-1 &, if  \ i = 3 \ for \ m = 4,
    \end{cases} \\
g^{''} (w_{i}^{0}, w_{0}^{j}) &= 
    \begin{cases}
        2j &, if  \ i = 2\\
        g^{'} (w_{i}^{0}, w_{0}^{j}) - 4mn &, if  \ i \neq 2
    \end{cases}\\
g^{'''} (w_{i}^{0}, w_{0}^{j}) &= 
    \begin{cases}
        4mn + (2j-1) &, if \  i = 2\\
        g^{'} (w_{i}^{0}, w_{0}^{j}) +1 &, if \  i \neq 2
    \end{cases}
		\end{split}
\end{equation}
From the above labeling schemes, we obtained the following vertex sums for every vertex $v.$ \\
The vertex sum of vertex $(w_{0}^{0})$ is,
\begin{equation} \nonumber
\begin{split}
\phi_{g^{'}} (w_{0}^{0}) = \phi_{g^{''}} (w_{0}^{0}) & =   5m^{2} n^{2} + mn,\\
\phi_{g^{'''}} (w_{0}^{0}) & =   5m^{2} n^{2} + n,
\end{split}
\end{equation}
For each $i, 1 \leq i \leq m$ and $1 \leq j \leq n, $ the vertex sum of the vertex $w_{i}^{j} \in \gamma^{'}$ is,
\begin{equation}  \nonumber
\begin{split}
\phi_{{g}^{'}} (w_{i}^{j}) &= 
    \begin{cases}
        8mn+4in-3n+6j &, if \ \text{i is odd}  \\
				16mn-4in+6j+n &, if \ \text{i is even},
    \end{cases}\\
		\phi_{{g}^{''}} (w_{i}^{j}) &= \phi_{{g}^{'}} (w_{i}^{j}) +4mn-1 \\
		\text{When} \ n = 2, \\
		\phi_{{g}^{'''}} (w_{i}^{j}) &= 
    \begin{cases}
        \phi_{{g}^{'}} (w_{i}^{j}) &, if \ i = 1  \\
				\phi_{{g}^{'}} (w_{i}^{j}) - 1 &,  if \ i \neq 1  
    \end{cases}\\
		\text{When} \ n \neq 2, \\
		\phi_{{g}^{'''}} (w_{i}^{j}) &= \phi_{{g}^{'}} (w_{i}^{j}) -1
\end{split}
\end{equation}
The vertex sum of the vertex $w_{m+i}^{j} \in \gamma^{'}$ gets their vertex sums as same as of the edge label $(w_{m+i}^{j},w_{i}^{0})$ under their labelled function.  \\
The vertex sum of the vertex $w_{i}^{0} \in \gamma^{''}$ for $1 \leq i \leq m$ is,\\
\begin{equation} \nonumber
\begin{split}
\phi_{g^{'}} (w_{i}^{0}) &= 
    \begin{cases}
		13mn^{2}+2n^{2}+n &, if \ i=1 \ \text{for} \ m=4 \\
		15mn^{2}-2n^{2}+n-2n^{2} \left\lceil \dfrac{m}{4}\right\rceil &, if \ i = 1 \ \text{for} \ m \neq 4 \\
		13mn^{2}+6n^{2}+n &, if \ i=3 \ \text{for} \ m=4 \\
		8mn^{2}+4in^{2}-2n^{2}+n &, if \ 2 \leq i \leq  2\left\lfloor \dfrac{m}{4}\right\rfloor \ \& \ \ i \ even \\
		8mn^{2}+4in^{2}+2n^{2}+n &, if \ 2\left\lfloor \dfrac{m}{4}\right\rfloor + 2 \leq i \leq m-2 \ \& \ \ i \ even \\
    16mn^{2}-4in^{2}+2n^{2}+n &, if \ 2 \left\lceil \dfrac{m}{4}\right\rceil + 1 \leq i \leq m-1 \ \& \ \ i \ odd \\
		9mn^{2}+2n^{2}+4n^{2}\left\lfloor \dfrac{m}{4}\right\rfloor + n &, if \  i =m, \\	
		16mn^{2}+4n^{2}-6n^{2}i+4n^{2}\left\lceil \dfrac{m}{4}\right\rceil + n &, if \ 3 \leq i \leq 2 \left\lceil \dfrac{m}{4}\right\rceil -1 for \ m \neq 4 				
    \end{cases} \\
\phi_{g^{''}} (w_{i}^{0}) &= 
		\begin{cases}
		4mn^{2} + 2in^{2} +2n^{2} + 2n &, if \ i=2 \\
		\phi_{g_{3}^{'}} (w_{i}^{0}) - 4mn^{2} + n &, if \ i \neq 2 
		\end{cases} \\
\text{When} \ n = 2, \\
		\phi_{g^{'''}} (w_{i}^{0}) &= 
		\begin{cases}
		\phi_{g^{'}} (w_{i}^{0}) &,  if \ i=2 \\
		\phi_{g^{'}} (w_{i}^{0})+n &, if \ i \neq 2 
		\end{cases} \\
\text{When} \ n \neq 2, \\			
		\phi_{g^{'''}} (w_{i}^{0}) &= \phi_{g^{'}} (w_{i}^{0}) + n
		\end{split}
\end{equation}
and the vertex sum of the vertex $w_{m+i}^{0} \in \gamma^{''}$ for each $i, 1 \leq i \leq m$ and $j, 1 \leq j \leq n,$
\begin{equation} \nonumber
\begin{split}
\phi_{g^{'}} (w_{m+i}^{0}) &= 
    \begin{cases}
        (i-1)n^{2}+n(n+1) &, if \ \text{i is odd}  \\
				mn^{2}+(m-i)n^{2}+n(n+1) &, if \ \text{i is even},
    \end{cases} \\
\phi_{g^{''}} (w_{m+i}^{0}) &= 4mn^{2} - n + \phi_{g_{{2}^{'}}} (w_{m+i}^{0})\\ 
\phi_{g^{'''}} (w_{m+i}^{0}) &= 
    \begin{cases}
        n^{2} &, \ if \ i = 1 \ \text{for} \ n \neq 2   \\
				n(n+1) &, \ if \ i = 1 \ \text{for} \ n = 2  \\
				\phi_{g_{{2}^{'}}} (w_{m+i}^{0}) &, if \ i \neq 1
    \end{cases}
		\end{split}
\end{equation}
Finally, the vertex sum of the vertex $w_{0}^{j} \in \gamma^{'''}$ for $j, 1 \leq j \leq n$ is,
\begin{equation} \nonumber
\begin{split}
\phi_{g^{'}} (w_{0}^{j}) &= 5m^{2}n-mn+(2j-1)m\\
\phi_{g^{''}} (w_{0}^{j})  &= \phi_{g^{'}} (w_{0}^{j}) - 4m^{2}n+1\\
\phi_{g^{'''}} (w_{0}^{j}) &= \phi_{g^{'}} (w_{0}^{j}) + m -1. 
\end{split}
\end{equation}
From the above vertex sums, we have the following observations,\\
(1) From the vertex $w_{i}^{j}$ for each $i, 1 \leq i \leq m$. We have,
\begin{equation} \nonumber
\phi_{{g}^{'}} (w_{1}^{j}) < \phi_{{g}^{'}} (w_{3}^{j}) < \cdots < \phi_{{g}^{'}} (w_{m-1}^{j}) < \phi_{{g}^{'}} (w_{m}^{j}) < \phi_{{g}^{'}} (w_{m-2}^{j}) < \phi_{{g}^{'}} (w_{m-4}^{j}) < \cdots < \phi_{{g}^{'}} (w_{2}^{j})
\end{equation}
further, for each $i, 1 \leq i \leq m$ there exist $j, 1 \leq j \leq n$ such that if $j_{1} < j_{2}$ we have $\phi_{{g}^{'}} (w_{i}^{j_{1}}) < \phi_{{g}^{'}} (w_{i}^{j_{2}})$.\\
(2) From the vertex $w_{m+i}^{j}$ for each $i, 1 \leq i \leq m$. We have,\\
If $m = 4,$
\begin{equation} \nonumber
\phi_{{g}^{'}} (w_{m+2}^{j}) < \phi_{{g}^{'}} (w_{m+4}^{j}) < \phi_{{g}^{'}} (w_{m+1}^{j}) < \phi_{{g}^{'}} (w_{m+3}^{j})
\end{equation}
If $m \geq 6,$
\begin{multline} \nonumber
\phi_{{g}^{'}} (w_{m+2}^{j}) < \phi_{{g}^{'}} (w_{m+4}^{j}) < \cdots < \phi_{{g}^{'}} (w_{m+ 2\left\lfloor \dfrac{m}{4}\right\rfloor}^{j}) < \phi_{{g}^{'}} (w_{2m}^{j}) < \phi_{{g}^{'}} (w_{m+ 2\left\lfloor \dfrac{m}{4}\right\rfloor + 2}^{j}) \\ <  \cdots < \phi_{{g}^{'}} (w_{2m-2}^{j})  < \phi_{{g}^{'}} (w_{2m-1}^{j}) < \cdots < \phi_{{g}^{'}} (w_{m+ 2\left\lceil \dfrac{m}{4}\right\rceil+1}^{j}) \\ < \phi_{{g}^{'}} (w_{m+1}^{j}) < \phi_{{g}^{'}} (w_{m+2\left\lceil \dfrac{m}{4}\right\rceil}^{j}) < \cdots < \phi_{{g}^{'}} (w_{m+3}^{j})
\end{multline}
further, for each $i, 1 \leq i \leq m$ there exist $j, 1 \leq j \leq n$ such that if $j_{1} < j_{2}$ we have $\phi_{{g}^{'}} (w_{m+i}^{j_{1}}) < \phi_{{g}^{'}} (w_{m+i}^{j_{2}})$.\\
(3) From the vertex $w_{i}^{0}$ for each $i, 1 \leq i \leq m,$. We have,  \\
If $m = 4,$
\begin{equation} \nonumber
\phi_{{g}^{'}} (w_{2}^{0}) < \phi_{{g}^{'}} (w_{4}^{0}) < \phi_{{g}^{'}} (w_{1}^{0}) < \phi_{{g}^{'}} (w_{3}^{0})
\end{equation}
If $m \geq 6,$
\begin{multline} \nonumber
\phi_{{g}^{'}} (w_{2}^{0}) < \phi_{{g}^{'}} (w_{4}^{0}) < \cdots < \phi_{{g}^{'}} (w_{2\left\lfloor \dfrac{m}{4}\right\rfloor}^{0}) < \phi_{{g}^{'}} (w_{m}^{0}) < \phi_{{g}^{'}} (w_{2\left\lfloor \dfrac{m}{4}\right\rfloor + 2}^{0}) <  \cdots \\ < \phi_{{g}^{'}} (w_{m-2}^{0}) < \phi_{{g}^{'}} (w_{m-1}^{0})  < \cdots < \phi_{{g}^{'}} (w_{2\left\lceil \dfrac{m}{4}\right\rceil+1}^{0}) \\ < \phi_{{g}^{'}} (w_{1}^{0}) < \phi_{{g}^{'}} (w_{2\left\lceil \dfrac{m}{4}\right\rceil}^{0}) < \cdots < \phi_{{g}^{'}} (w_{3}^{0})
\end{multline}
(4) From the vertex $w_{m+i}^{0}$ for each $i, 1 \leq i \leq m,$. We have,  \\
\begin{multline} \nonumber
\phi_{{g}^{'}} (w_{m+1}^{0}) < \phi_{{g}^{'}} (w_{m+3}^{0}) < \cdots < \phi_{{g}^{'}} (w_{2m-1}^{0}) < \phi_{{g}^{'}} (w_{2m}^{0}) < \phi_{{g}^{'}} (w_{2(m-1)}^{0}) \\ < \phi_{{g}^{'}} (w_{2(m-2)}^{0}) < \cdots <  \phi_{{g}^{'}} (w_{m+4}^{0}) < \phi_{{g}^{'}} (w_{m+2}^{0})
\end{multline}
(5) For each $j, 1 \leq j \leq n,$ the vertex $w_{0}^{j}$ have,
\begin{equation} \nonumber
\phi_{{g}^{'}} (w_{0}^{1}) < \phi_{{g}^{'}} (w_{0}^{2}) < \cdots < \phi_{{g}^{'}} (w_{0}^{n})
\end{equation}
These observations also holds true for ${g}^{''}$ and ${g}^{'''}$.  Therefore, it is clear that no two vertices belongs to the same set of vertices, $\gamma^{'}, \gamma^{''}, \gamma^{'''} $ are equal. And from the above observations and the defined vertex sums, we have the following:
\begin{itemize}
	\item The vertex sum of the vertex $w_{0}^{0}$ is greater than all the other vertex sums.
	\item If the vertex having a odd sum then we have, \\
	$\phi_{{g}^{'}} (w_{2}^{n}) < \phi_{{g}^{'}} (w_{2}^{0}), \phi_{{g}^{''}} (w_{m+2}^{n}) < \phi_{{g}^{''}} (w_{0}^{1}) < \phi_{{g}^{''}} (w_{0}^{n}) < \phi_{{g}^{''}} (w_{m+1}^{0}).$\\
	If $n=2, \phi_{{g}^{'''}} (w_{2}^{n}) < \phi_{{g}^{'''}} (w_{0}^{1}).$\\
	If $n \neq 2, \phi_{{g}^{'''}} (w_{m+3}^{n}) < \phi_{{g}^{'''}} (w_{1}^{1}) < \phi_{{g}^{'''}} (w_{2}^{n}) < \phi_{{g}^{'''}} (w_{0}^{1})$.	
\item If the vertex having a even sum then we have, \\
$\phi_{{g}^{'}} (w_{m+2}^{0}) < \phi_{{g}^{'}} (w_{0}^{1}), \phi_{{g}^{''}} (w_{2}^{n}) < \phi_{{g}^{''}} (w_{2}^{0})$.\\
If $n=2, \phi_{{g}^{'''}} (w_{1}^{n}) > \phi_{{g}^{'''}} (w_{m+2}^{0})$.\\
If $n \neq 2, \phi_{{g}^{'''}} (w_{m+2}^{n}) < \phi_{{g}^{'''}} (w_{2}^{0})$.
\end{itemize}
Hence, for any two distinct vertices we obtained distinct vertex sum. Therefore, $G$ is antimagic. \\
From lemma. \ref{Helmn=1}, theorem. \ref{Helmmodd} and \ref{Helmmeven} we have the following theorem.
\begin{theorem} \label{Helmmain}
The tensor product of Helm $W_m, m \geq 3$ and star $K_{1,n}, n \geq 2$  admits antimagic labeling.
\end{theorem}
\subsection{Antimagicness for tensor product of flower and star}
\noindent  Let the vertex set of the tensor product of $Fl_{m} \times K_{1,n} = G$ is,
\begin{equation} \nonumber
V(G) = \left\{(u_i, v_j) = w_{i}^{j}, 0 \leq i \leq 2m,  0 \leq j \leq n\right\}
\end{equation}
Note that, the tensor product of helm and star is the subgraph of the tensor product of flower and star. So, we can write the vertex set of $G$ can be written using the vertex set of tensor product of helm and star, 
\begin{equation} \nonumber
V(G) = \left\{w_{0}^{0}\right\} \cup \gamma^{'} \cup  \gamma^{''} \cup  \gamma^{'''}
\end{equation}
where, $\gamma^{'}, \gamma^{''}$ and $\gamma^{'''}$ are defined in section $3.2$.\\
The edge set of $G$ is defined as,
\begin{align*} 
E(G) = \eta^{'} \cup \eta^{''} \cup \eta^{'''}
\end{align*}
where,
$\eta^{'} = \delta^{'} \cup \left\{(w_{0}^{0}, w_{m+i}^{j}), 1 \leq i \leq m, 1 \leq j \leq n\right\}, \eta^{''} = \delta^{''} $ and
\\
$\eta^{'''}  = \delta^{'''} \cup \left\{(w_{m+i}^{0},w_{0}^{j}), 1 \leq i \leq m, 1 \leq j \leq n\right\}$
where $\delta^{'}, \delta^{''}$ and $\delta^{'''}$ are defined in section $3.2.$ \\
Note that $\left|E(G)\right| = 8mn.$ We begin the proof with the following lemma.
\begin{lemma} \label{flowern=1}
The graph $G= Fl_{m} \times K_{1,1}$ is antimagic for $m \geq 3.$
\end{lemma}
\begin{proof}
We define the labeling $f_{2}: E(G) \rightarrow \left\{1,2, \cdots, 8m\right\}$ with the help of the labeling  function $f_{1}$ defined in proof of lemma \ref{Helmn=1} . Now, let us label the edges of the graph $G$ as defined as follows,\\
The edges of $\eta^{'}$ is labelled as,
\begin{align*} \nonumber
f_{2} (w_{0}^{0}, w_{i}^{1})=& 2m+ f_{1} (w_{0}^{0},w_{i}^{1})   \\
f_{2} (w_{0}^{0}, w_{m+i}^{1})=& 2m+ f_{1} (w_{i}^{0},w_{m+i}^{1}) - 1
\end{align*}
The edges of $\eta^{''}$ is labelled as,
\begin{align*} \nonumber
f_{2} (w_{i}^{1}, w_{i+1}^{0})=& 2m+ f_{1} (w_{i}^{1},w_{i+1}^{1})  \\
f_{2} (w_{1}^{1}, w_{m}^{0})=& 2m+ f_{1} (w_{i}^{1},w_{i+1}^{0}) \\
f_{2} (w_{m}^{1}, w_{1}^{0})=& 8m-3 \\
f_{2} (w_{i}^{0}, w_{i+1}^{1})=& 2m+ f_{1} (w_{i}^{1},w_{i+1}^{1})\\
f_{2} (w_{i}^{1}, w_{m+i}^{0})=& 2i \\
f_{2} (w_{i}^{0}, w_{m+i}^{1})=& f_{1} (w_{i}^{0},w_{m+i}^{1} ) - 1
\end{align*}
The edges of $\eta^{'''}$ is labelled as,
\begin{align*} \nonumber
f_{2} (w_{m+i}^{0}, w_{0}^{1})=& 2m+ 2i \\
f_{2} (w_{i}^{0}, w_{0}^{1})=& 2m+ f_{1} (w_{i}^{0},w_{0}^{1} ) 
\end{align*}
The observation $(1)$ as given in Lemma. (\ref{Helmn=1}), the same holds true under the function $f_{2}$. Clearly, the above vertex sums are distinct due to the following reasoning:
\begin{enumerate}
	\item The degree two vertices $w_{m+i}^{j}$ and $w_{m+i}^{0}$  is lesser from all the other vertices and it is distinct between themselves (the vertex sum of these vertices ranges from the set $\left\{2m+2,2m+4,\cdots,6m\right\}$).
	\item If $m$ is odd and having odd sum then we have,\\
	For $m = 3, \phi_{f_{2}} (w_{3}^{0}) <  \phi_{f_{2}} (w_{2}^{0}) <  \phi_{f_{2}} (w_{1}^{0}) <  \phi_{f_{2}} (w_{0}^{0}) $.\\
	For $m \geq 5,  \phi_{f_{2}} (w_{m-1}^{0}) <  \phi_{f_{2}} (w_{0}^{0})$.
	\item If $m$ is odd and having even sum then we have,\\
	For $m = 3, \phi_{f_{2}} (w_{m+1}^{1 }) < \phi_{f_{2}} (w_{2m}^{0 }) < \phi_{f_{2}} (w_{1}^{1 }) < \phi_{f_{2}} (w_{3}^{1 }) < \phi_{f_{2}} (w_{0}^{1 })$.\\
	For $m \geq 5,  \phi_{f_{2}} (w_{2m-1}^{1})  < \phi_{f_{2}} (w_{2m}^{0}) < \phi_{f_{2}} (w_{1}^{1}) < \phi_{f_{2}} (w_{m}^{1})  < \phi_{f_{2}} (w_{0}^{1})$.
	\item If $m$ is even and having even sum then we have,
	$\phi_{f_{2}} (w_{2m-1}^{1}) < \phi_{f_{2}} (w_{2m}^{0}) < \phi_{f_{2}} (w_{1}^{1}) < \phi_{f_{2}} (w_{m}^{1}) < \phi_{f_{2}} (w_{0}^{0}) < \phi_{f_{2}} (w_{0}^{1})$ (and $\phi_{f_{2}} (w_{i}^{0})$ is the only odd sum).	
\end{enumerate}
Thus, any two distinct vertices of $G$ getting distinct vertex sums. Therefore, $G$ is antimagic.
\end{proof}
Next, we prove the main cases for the values $m \geq 3$ and $n \geq 2$.
\begin{theorem}  \label{flowermodd}
The tensor product of Flower $Fl_m, m \geq 3$ and star $K_{1,n}$ where $m$ is odd and $n \geq 2$  admits antimagic labeling.
\end{theorem}
\begin{theorem}  \label{flowermeven}
The tensor product of Flower $Fl_m, m \geq 4$ and star $K_{1,n}$ where $m$ is even and $n \geq 2$  admits antimagic labeling.
\end{theorem}
Before proving theorem \ref{flowermodd} and theorem \ref{flowermeven}, we classify the tensor product of graph $G$ from Flower $Fl_{m}, m \geq 3$ and star $K_{1,n}, n \geq 2$ into three cases based on the values of $m$ and $n$.
\begin{equation}   \nonumber
G \cong
     \begin{cases}
        G^{'} &, if \ m \geq n \ \& \ n \ \text{is odd}   \\
				G^{''} &, if \ m < n \ \& \ n \ \text{is odd}   \\
        G^{'''} &, if  \ n \ \text{is even}, 
    \end{cases}
\end{equation}
In the proof of theorem \ref{flowermodd} and \ref{flowermeven}, we use the labelling of $G^{'}, G^{''}$ and $G^{'''}$ as $h^{'}, h^{''}$ and $h^{'''}$ respectively.\\
\textbf{Proof of theorem \ref{flowermodd}: } The labeling of the edges of graph $G$ are given in the following steps and we defined the labeling with the help of the labeling defined in the proof of theorem. \ref{Helmmodd}.\\
\emph{Step 1: } Labeling the edges of $(w_{0}^{0},w_{i}^{j}) \in \eta^{'}$ for $1 \leq i \leq m$ and $1 \leq j \leq n$ is defined as follows:
\begin{equation}   \nonumber
\begin{split}
h^{'} (w_{0}^{0},w_{i}^{j}) &= 2mn+g^{'} (w_{0}^{0},w_{i}^{j})\\
h^{''} (w_{0}^{0},w_{i}^{j}) &= h^{'} (w_{0}^{0},w_{i}^{j})   - 1 \\
h^{'''} (w_{0}^{0},w_{i}^{j}) &=
    \begin{cases}
        h^{'} (w_{0}^{0},w_{i}^{j}) &, if \  m \geq 5  \\
				 h^{'} (w_{0}^{0},w_{i}^{j}) - 1 &, if \  m = 3 
    \end{cases}
		\end{split}
\end{equation}
and label the edges of $(w_{0}^{0},w_{m+i}^{j})$ for each $i, 1 \leq i \leq m$ and $j, 1 \leq j \leq n,$
\begin{equation} \nonumber
\begin{split}
h^{'} (w_{0}^{0},w_{m+i}^{j}) & =  \begin{cases}
        (i-2)n+2j-1 &, if \ i \ \text{is even}   \\
				 n(m-1) + 2j-1 &, if \  i = m \\
				n(m-1) + 2n+(2j-1)+(i-1)n &, if \ i \ \text{is odd}
    \end{cases}\\
h^{''} (w_{0}^{0},w_{m+i}^{j}) &= h^{'} (w_{0}^{0},w_{m+i}^{j}) + 1\\
h^{'''} (w_{0}^{0},w_{m+i}^{j}) &=
    \begin{cases}
        h^{'} (w_{0}^{0},w_{m+i}^{j}) &, if \ m \geq 5   \\
				 2j &, if \  i = 2 \ \& \ m = 3\\
				h^{'} (w_{0}^{0},w_{m+i}^{j}) &, if \ i \neq 2 \ \& \ m = 3
    \end{cases}
\end{split}
\end{equation}
\emph{Step 2: } Labeling the edges of $\eta^{''}$ as follows:\\
First, label the edges of $(w_{i}^{j}, w_{m+i}^{0})$ for each $i, 1 \leq i \leq m$ and $1 \leq j \leq n$,\\
\begin{equation}   \nonumber
\begin{split}
h^{'} (w_{i}^{j}, w_{m+i}^{0}) &=  2mn + g^{'} (w_{i}^{j}, w_{m+i}^{0}) \\
h^{''} (w_{i}^{j}, w_{m+i}^{0}) &= 4mn + 1 + h^{'} (w_{i}^{j}, w_{m+i}^{0}) \\
h^{'''} (w_{i}^{j}, w_{m+i}^{0}) & = \begin{cases}
        h^{'} (w_{i}^{j}, w_{m+i}^{0}) &, if \  m \geq 5   \\
        4mn+1+ h^{'} (w_{i}^{j}, w_{m+i}^{0}) &, if \  m = 3 
    \end{cases}
		\end{split}
\end{equation}
Next, label the edges of $(w_{m+i}^{j}, w_{i}^{0})$ for each $i, 1 \leq i \leq m$ and $1 \leq j \leq n$,
\begin{equation}   \nonumber
\begin{split}
h^{'} (w_{m+i}^{j}, w_{i}^{0}) &=  2mn + g^{'} (w_{m+i}^{j}, w_{i}^{0}) \\
 h^{''} (w_{m+i}^{j}, w_{i}^{0}) &=  h^{'} (w_{m+i}^{j}, w_{i}^{0}) - 1 \\
    h^{'''} (w_{m+i}^{j}, w_{i}^{0}) &=
		\begin{cases}
        h^{'} (w_{m+i}^{j}, w_{i}^{0}) &, if \ m \geq 5  \\
				2mn+2j &, if \ i = 2 \ \& \ m = 3 \\
				 h^{'} (w_{m+i}^{j}, w_{i}^{0}) - 1 &, if \ i \neq 2 \ \& \ m = 3 
    \end{cases}
\end{split}
\end{equation}
Finally, we label the remaining edges of $\eta^{''}$ for each $i, 1 \leq i \leq m$ and for $j, 1 \leq j \leq n,$
\begin{equation}   \nonumber
\begin{split}
h^{'} (w_{i}^{0}, w_{i+1}^{j}) &= 2mn + g^{'} (w_{i}^{0}, w_{i+1}^{j}) \\
h^{''} (w_{i}^{0}, w_{i+1}^{j}) &= h^{'} (w_{i}^{0}, w_{i+1}^{j}) \\
h^{'''} (w_{i}^{0}, w_{i+1}^{j}) &= h^{'} (w_{i}^{0}, w_{i+1}^{j}) \\
h^{'} (w_{i}^{j}, w_{i+1}^{0}) &= 2mn + g^{'} (w_{i}^{j}, w_{i+1}^{0}) \\
h^{''} (w_{i}^{j}, w_{i+1}^{0}) &= h^{'} (w_{i}^{j}, w_{i+1}^{0}) \\
h^{'''} (w_{i}^{j}, w_{i+1}^{0}) &= h^{'} (w_{i}^{j}, w_{i+1}^{0}) \\
h^{'} (w_{1}^{j}, w_{m}^{0}) &= 4mn+j \\
h^{''} (w_{1}^{j}, w_{m}^{0}) &= h^{'} (w_{1}^{j}, w_{m}^{0}) \\
h^{'''} (w_{1}^{j}, w_{m}^{0}) &= h^{'} (w_{1}^{j}, w_{m}^{0}) \\
h^{'} (w_{m}^{j}, w_{1}^{0}) &= 5mn+j \\
h^{''} (w_{m}^{j}, w_{1}^{0}) &= h^{'} (w_{m}^{j}, w_{1}^{0}) \\
h^{'''} (w_{m}^{j}, w_{1}^{0}) &= h^{'} (w_{m}^{j}, w_{1}^{0}) 	
\end{split}
\end{equation}
\emph{Step 3: }  We label the edges of $(w_{i}^{0}, w_{0}^{j}) \in \eta^{'''}$ as follows:\\
Label the edges of $(w_{i}^{0}, w_{0}^{j}) \in \eta^{'''}$ for each $i, 1 \leq i \leq m$ and $j, 1 \leq j \leq n,$
\begin{equation}   \nonumber
\begin{split}
h^{'} (w_{i}^{0}, w_{0}^{j}) &= 2mn+ g^{'} (w_{i}^{0}, w_{0}^{j}) \\
h^{''} (w_{i}^{0}, w_{0}^{j}) &= h^{'} (w_{i}^{0}, w_{0}^{j}) - 6mn \\
h^{'''} (w_{i}^{0}, w_{0}^{j}) &= \begin{cases}
		    h^{'} (w_{i}^{0}, w_{0}^{j}) &, if \  m \geq 5 \\
				h^{'} (w_{i}^{0}, w_{0}^{j}) - 6mn &, if \ i \neq  2  \ \& \ m = 3 \\
				h^{'} (w_{i}^{0}, w_{0}^{j}) - 6mn + 1 &, if \ i \neq  2  \ \& \ m = 3 
    \end{cases}
\end{split} 
\end{equation}
Finally we label the remaining edges of $(w_{m+i}^{0},w_{0}^{j})$ for each $i, 1 \leq i \leq m$ and for $j, 1 \leq j \leq n,$
\begin{equation}   \nonumber
\begin{split}
h^{'} (w_{m+i}^{0}, w_{0}^{j}) &= 
    \begin{cases}
		    (i-1)n+2j &, if \ i \ \text{is odd} \\
				mn+n+(i-2)n+2j &, if \ i \ \text{is even}
    \end{cases}\\
h^{''} (w_{m+i}^{0}, w_{0}^{j}) &= 2mn+ h^{'} (w_{m+i}^{0}, w_{0}^{j})\\
h^{'''} (w_{m+i}^{0}, w_{0}^{j}) &=  
\begin{cases}
		    h^{'} (w_{m+i}^{0}, w_{0}^{j}) &, if \ m \geq 5 \\
				 h^{'} (w_{m+i}^{0}, w_{0}^{j}) + 2mn - 1 &, if \ i = 1 \ \& \ m =3 \\
				 h^{'} (w_{m+i}^{0}, w_{0}^{j}) + 2mn, if \ i \neq 1 \ \& \ m =3 
    \end{cases}\\
\end{split}
\end{equation}
From the above labeling schemes, we obtained the following vertex sums for every vertex $v$.\\
The vertex sum of vertex $(w_{0}^{0})$ is,
\begin{equation} \nonumber
\begin{split}
\phi_{h^{'}} (w_{0}^{0}) = \phi_{h^{''}} (w_{0}^{0}) & =   8m^{2} n^{2} + mn\\
\phi_{h^{'''}} (w_{0}^{0}) & =   
\begin{cases}
		    \phi_{h^{'}} (w_{0}^{0}) &, if \ m \geq 5 \\
				\phi_{h^{'}} (w_{0}^{0}) - mn + n &, if \ m = 3
    \end{cases}
\end{split}
\end{equation}
For each $i, 1 \leq i \leq m$ and $1 \leq j \leq n, $ the vertex sum of the vertex $(w_{i}^{j})$ is,
\begin{equation}  \nonumber
\begin{split}
\phi_{{h}^{'}} (w_{i}^{j}) &= 8mn+ \phi_{{g}^{'}} (w_{i}^{j}) \\
\phi_{{h}^{''}} (w_{i}^{j}) &= 4mn+ \phi_{{h}^{'}} (w_{i}^{j}) \\
\phi_{{h}^{'''}} (w_{i}^{j}) &= 
\begin{cases}
		    \phi_{{h}^{'}} (w_{i}^{j}) &, if \ m \geq 5 \\
				\phi_{{h}^{'}} (w_{i}^{j}) + 4mn &, if \ m = 3
 \end{cases}
\end{split}
\end{equation}
and the vertex sum of the vertex $w_{m+i}^{j}$ is,
\begin{equation}  \nonumber
\begin{split}
\phi_{{h}^{'}} (w_{m+i}^{j}) &= 
\begin{cases}
		    2mn+4j-1+2(i-2)n &, if \ i \ \text{is even} \\
				2(m-1)n+4j+2mn-1 &, if \ i = m \\
				2n(m+i)+4j-1+2mn &, if \ i \ \text{is odd}
 \end{cases}\\
\phi_{{h}^{''}} (w_{m+i}^{j}) &= \phi_{{h}^{'}} (w_{m+i}^{j}) \\
\phi_{{h}^{'''}} (w_{m+i}^{j}) &= 
\begin{cases}
		    \phi_{{h}^{'}} (w_{m+i}^{j}) &, if \ m \geq 5 \\
				\phi_{{h}^{'}} (w_{m+i}^{j}) + 1 &, if \ i = 2 \ \& \ m =3 \\
				\phi_{{h}^{'}} (w_{i}^{j}) + 4mn &, if \ i \neq 2 \ \& \ m =3 \\
 \end{cases}
\end{split}
\end{equation}
The vertex sum of the vertex $(w_{i}^{0})$ for $1 \leq i \leq m$ is,
\begin{equation}   \nonumber
\begin{split}
\phi_{{h^{'}}} (w_{i}^{0}) &= 8mn^{2} + \phi_{{g^{'}}} (w_{i}^{0}) \\
\phi_{{h^{''}}} (w_{i}^{0}) &= \phi_{{h^{'}}} (w_{i}^{0}) - 6mn^{2}  - n \\
\phi_{{h^{'''}}} (w_{i}^{0}) &=
    \begin{cases}
        \phi_{{h^{'}}} (w_{i}^{0}) &, if \ m \geq 5 \\
				\phi_{{h^{'}}} (w_{i}^{0}) - 6mn^{2} &, if \ m =3
    \end{cases}
		\end{split}
\end{equation}
and the vertex sum of the vertex $(w_{m+i}^{0})$ for each $i, 1 \leq i \leq m$ and $j, 1 \leq j \leq n,$
\begin{equation}  \nonumber
\begin{split}
\phi_{{h^{'}}} (w_{m+i}^{0}) &= 
    \begin{cases}
        2mn^{2} + in^{2} + n(n+1) + n^{2} (i-1) &, if \ \text{i is odd}  \\
				2mn^{2} + 2n^{2} (m+i) + n &, if \ \text{i is even},
    \end{cases} \\
\phi_{{h^{''}}} (w_{m+i}^{0}) &= 6mn^{2} + n + \phi_{{h^{'}}} (w_{m+i}^{0}) \\
\phi_{{h^{'''}}} (w_{m+i}^{0}) &= 
  \begin{cases}
        \phi_{{h^{'}}} (w_{m+i}^{0}) &, if \ m \geq 5 \\
				\phi_{{h^{'}}} (w_{m+i}^{0}) + 6mn^{2} &, if \ i = 1 \ \& \ m =3 \\
	\phi_{{h^{'}}} (w_{m+i}^{0}) + 6mn^{2} + n &, if \ i \neq 1 \ \& \ m =3
    \end{cases}
		\end{split}
\end{equation}
Finally, the vertex sum of the vertex $w_{0}^{j}$ for $j, 1 \leq j \leq n$ is,
\begin{equation} \ \nonumber
\begin{split}
\phi_{h^{'}} (w_{0}^{j}) &= 8m^{2} n -2mn + m(4j-1) \\
\phi_{h^{''}} (w_{0}^{j}) &= 4m^{2} n +  \phi_{h^{'}} (w_{0}^{j}) \\
\phi_{h^{''}} (w_{0}^{j}) &= 
    \begin{cases}
        \phi_{{h^{'}}} (w_{0}^{j}) &, \ if \ m \geq 5  \\
				\phi_{{h^{'}}} (w_{0}^{j}) -4m^{2}n + 1   &, \ if \ m = 3
    \end{cases}
\end{split}
\end{equation}
The observation $ (1),(2)$ and $(3)$ as given in theorem. \ref{Wheelmodd} and the observation $(1)$ and $(2) $ from theorem. \ref{Helmmodd}) \ is also true under the function $h^{'},h^{''}$ and $h^{'''}$.
Therefore, it is clear that no two vertices belongs to the same set of vertices are equal. And from the above observations and the defined vertex sums, we have the following:
\begin{itemize}
	\item The vertex sum of the vertex $w_{0}^{0}$ is greater than all the other vertex sums.
	\item If the vertex having a odd sum then we have, \\
$\phi_{{h}^{'}} (w_{2m-2}^{n}) <  \phi_{{h}^{'}} (w_{m+1}^{0}) < \phi_{{h}^{'}} (w_{2m-1}^{0}) < \phi_{{h}^{'}} (w_{0}^{1})  < \phi_{{h}^{'}} (w_{0}^{n}), \phi_{{h}^{''}} (w_{2m-2}^{n}) < \phi_{{h}^{''}} (w_{0}^{1}) < \phi_{{h}^{''}} (w_{0}^{n}) < \phi_{{h}^{''}} (w_{2}^{0}) < \phi_{{h}^{''}} (w_{m-2}^{0})$.\\
When $m \geq 5, \phi_{{h}^{'''}} (w_{2m-2}^{n}) < \phi_{{h}^{'''}} (w_{1}^{1}) < \phi_{{h}^{'''}} (w_{m-1}^{n}) < \phi_{{h}^{'''}} (w_{0}^{1})$ \\
When $m=3$, the vertex $w_{i}^{j}$ for $1 \leq i \leq m$ and for $j, 1 \leq j \leq n$ is the only odd sum.
\item If the vertex having a even sum then we have, \\
$\phi_{{h}^{'}} (w_{m-1}^{n}) < \phi_{{h}^{'}} (w_{2}^{0}), \phi_{{h}^{''}} (w_{m-1}^{n}) < \phi_{{h}^{''}} (w_{m+1}^{0}).  $\\
When $m \geq 5, w_{i}^{0}$ is the only odd sum.\\
When $m = 3,  \phi_{{h}^{'''}} (w_{m+1}^{n}) < \phi_{{h}^{'''}} (w_{0}^{1}) < \phi_{{h}^{'''}} (w_{0}^{n}) < \phi_{{h}^{'''}} (w_{2}^{0}) < \phi_{{h}^{'''}} (w_{3}^{0}) < \phi_{{h}^{'''}} (w_{1}^{0})$.
\end{itemize} 
Hence, for any two distinct vertices we obtained distinct vertex sum. Therefore, $G$ is antimagic.\\
\\
\textbf{Proof of theorem \ref{flowermeven}.} The labeling of edges of graph $G$ are given in the following steps and we defined the labeling with the help of the labeling defined in the proof of theorem. (\ref{Helmmeven}) \\
\emph{Step 1: } Labeling the edges of $(w_{0}^{0}, w_{i}^{j}) \in \eta^{'}$ for each $i, 1 \leq i \leq m$ and for $j, 1 \leq j \leq n$ is defined as follows:
\begin{equation}   \nonumber
\begin{split}
h^{'} (w_{0}^{0},w_{i}^{j}) &= g^{'} (w_{0}^{0},w_{i}^{j}) + 2mn \\
h^{''} (w_{0}^{0},w_{i}^{j}) &=  h^{'} (w_{0}^{0},w_{i}^{j}) - 1 \\
h^{'''} (w_{0}^{0},w_{i}^{j}) &=  h^{'} (w_{0}^{0},w_{i}^{j})
		\end{split}
\end{equation}
and label the edges of $(w_{0}^{0},w_{m+i}^{j})$ for each $i, 1 \leq i \leq m$ and each $j, 1 \leq j \leq n,$
\begin{equation} \nonumber
\begin{split}
h^{'} (w_{0}^{0}, w_{m+i}^{j}) &= 
    \begin{cases}
        2j-1+n(i-2) ,& if \ 2 \leq i \leq  2\left\lfloor \dfrac{m}{4}\right\rfloor \ \& \ i \ even \\
        2n\left\lfloor \dfrac{m}{4}\right\rfloor + 2j-1,& if i =m, \\
				ni+2j-1, & if \ 2\left\lfloor \dfrac{m}{4}\right\rfloor + 2 \leq i \leq m-2 \ \& \ \ i \ even \\
				n(2m-1-i)+2j-1,& if \ 2 \left\lceil \dfrac{m}{4}\right\rceil + 1 \leq i \leq m-1 \ \& \ \ i \ odd \\
				2mn-2n\left\lceil \dfrac{m}{4}\right\rceil + 2j -1,& if \ i=1 for  \ m \neq 4\\
				3mn-4n\left\lceil \dfrac{m}{4}\right\rceil + (3-i)n+2j-1,& if \  3 \leq  i \leq 2\left\lceil \dfrac{m}{4}\right\rceil -1 \ \& \ \ i \ odd \ for \ m \neq 4\\
				4n+2j-1,& if i = 1 \ for \ m = 4\\
				6n+2j-1, & if i=3 \ for \ m=4
    \end{cases}\\
h^{''} (w_{0}^{0}, w_{m+i}^{j}) &= 	2mn+1+	h^{'} (w_{0}^{0}, w_{m+i}^{j}) \\
h^{'''} (w_{0}^{0}, w_{m+i}^{j}) &=  
\begin{cases}
        2j -1 ,& if \ i = 2 \ \& \ n = 2 \\
				2j  ,& if \ i = 2 \ \& \ n \neq 2 \\
        h^{'} (w_{0}^{0}, w_{m+i}^{j}) ,& \text{otherwise} 
    \end{cases}
		\end{split}
\end{equation}
\emph{Step 2: } Labeling the edges of $\eta^{''}$ as follows:\\
First, label the edges of $(w_{i}^{j}, w_{m+i}^{0})$ for each $i, 1 \leq i \leq m$ and $1 \leq j \leq n$,\\
\begin{equation}   \nonumber
\begin{split}
h^{'} (w_{i}^{j}, w_{m+i}^{0}) &= 2mn+ g^{'} (w_{i}^{j}, w_{m+i}^{0}) \\     
h^{''} (w_{i}^{j}, w_{m+i}^{0}) &= 4mn + h^{'} (w_{i}^{j}, w_{m+i}^{0}) \\
h^{'''} (w_{i}^{j}, w_{m+i}^{0}) &= h^{'} (w_{i}^{j}, w_{m+i}^{0}) - 1
	\end{split}
\end{equation}
Next, label the edges of $(w_{m+i}^{j}, w_{i}^{0})$ for each $i, 1 \leq i \leq m$ and $1 \leq j \leq n$,
 \begin{equation} \nonumber
\begin{split}
h^{'} (w_{m+i}^{j}, w_{i}^{0}) &= 2mn + g^{'} (w_{m+i}^{j}, w_{i}^{0}) \\
h^{''} (w_{m+i}^{j}, w_{i}^{0}) &=
    \begin{cases}
		 2j ,& if  \ i =2 \\
		h^{'} (w_{m+i}^{j}, w_{i}^{0}),& if \  i \neq 2
    \end{cases}\\		
h^{'''} (w_{m+i}^{j}, w_{i}^{0}) &= h^{'} (w_{m+i}^{j}, w_{i}^{0}) + 1	
		\end{split}
\end{equation}
Finally, we label the remaining edges of $\eta^{''}$ for each $i, 1 \leq i \leq m$ and for $j, 1 \leq j \leq n,$
\begin{equation} \nonumber
\begin{split}
h^{'} (w_{1}^{j}, w_{m}^{0}) &= 4mn+j \\
h^{''} (w_{1}^{j}, w_{m}^{0}) &= h^{'} (w_{1}^{j}, w_{m}^{0}) \\
h^{'''} (w_{1}^{j}, w_{m}^{0}) &= h^{'} (w_{1}^{j}, w_{m}^{0}) \\
h^{'} (w_{m}^{j}, w_{1}^{0}) &= 5mn+j
\end{split}
\end{equation} 
\begin{equation} \nonumber
\begin{split}
h^{''} (w_{m}^{j}, w_{1}^{0}) &= h^{'} (w_{m}^{j}, w_{1}^{0}) \\
h^{'''} (w_{m}^{j}, w_{1}^{0}) &= h^{'} (w_{m}^{j}, w_{1}^{0}) \\
h^{'} (w_{i}^{j}, w_{i+1}^{0}) &= 
\begin{cases}
        2mn+(2m+i)n+j &, if \ \text{i is odd}  \\
				2mn+(4m-i)n+j &, if \ \text{i is even}
    \end{cases} \\
h^{''} (w_{i}^{j}, w_{i+1}^{0}) &= h^{'} (w_{i}^{j}, w_{i+1}^{0})	 \\
h^{'''} (w_{i}^{j}, w_{i+1}^{0}) &= h^{'} (w_{i}^{j}, w_{i+1}^{0})		 \\
h^{'} (w_{i}^{0}, w_{i+1}^{j}) &=  
    \begin{cases}
        2mn + (4m-i)n + j &, if \ \text{i is odd}  \\
				2mn + (2m+i)n + j &, if \ \text{i is even}
    \end{cases} \\
h^{''} (w_{i}^{0}, w_{i+1}^{j}) &=  h^{'} (w_{i}^{0}, w_{i+1}^{j}) \\
h^{'''} (w_{i}^{0}, w_{i+1}^{j}) &=  h^{'} (w_{i}^{0}, w_{i+1}^{j})
\end{split}
\end{equation}
\emph{Step 3: }  
We label the edges of $\eta^{'''} $ nas follows:\\
The edges of $(w_{i}^{0}, w_{0}^{j}) \in \eta^{'''}$ for each $i, 1 \leq i \leq m$ and $1 \leq j \leq n, $
\begin{equation} \nonumber
\begin{split}
h^{'} (w_{i}^{0}, w_{0}^{j}) &=  2mn+ g^{'} (w_{i}^{0}, w_{0}^{j}) \\
h^{''} (w_{i}^{0}, w_{0}^{j}) &= 
    \begin{cases}
        2j - 1 ,& if \  i = 2\\
        h^{'} (w_{i}^{0}, w_{0}^{j}) - 6mn + 1,& if  \ i \neq 2
    \end{cases}\\
h^{'''} (w_{i}^{0}, w_{0}^{j}) &= h^{'} (w_{i}^{0}, w_{0}^{j})
		\end{split}
\end{equation}
Finally, we label the remaining edges of $(w_{m+i}^{0},w_{0}^{j}) \in \eta^{'''}$ for each $i, 1 \leq i \leq m$ and for $j, 1 \leq j \leq n$,
\begin{equation} \nonumber
\begin{split}
h^{'} (w_{m+i}^{0}, w_{0}^{j}) &= 
\begin{cases}
         2j+(i-1)n,& \text{if i is odd} \\
				mn+2j+(m-i)n,&  \text{if i is even}
\end{cases}\\
h^{''} (w_{m+i}^{0}, w_{0}^{j}) &= 	2mn+ 	h^{'} (w_{m+i}^{0}, w_{0}^{j}) -1 \\
h^{'''} (w_{m+i}^{0}, w_{0}^{j}) &= 
		\begin{cases}
        2j -1 ,& if  \ i = 1 \ \& \ n \neq 2  \\
				2j ,& if  \ i = 1 \ \& \ n = 2  \\
        h^{'} (w_{m+i}^{0}, w_{0}^{j}), & \text{Otherwise.}
    \end{cases}
\end{split}		
\end{equation}
From the above labeling schemes, we obtained the following vertex sums for every vertex $v.$ \\
The vertex sum of vertex $(w_{0}^{0})$ is,
\begin{equation} \nonumber
\begin{split}
\phi_{h^{'}} (w_{0}^{0}) = \phi_{h^{''}} (w_{0}^{0}) & =   8m^{2} n^{2} + mn,\\
\phi_{h^{'''}} (w_{0}^{0}) & =   
\begin{cases}
         \phi_{h^{'}} (w_{0}^{0}),& if  \  n = 2 \\
				\phi_{h^{'}} (w_{0}^{0}) + n,& if  \  n > 2 
\end{cases}
\end{split}
\end{equation}
For each $i, 1 \leq i \leq m$ and $1 \leq j \leq n, $ the vertex sum of the vertex $(w_{i}^{j})$ is,
\begin{equation}  \nonumber
\begin{split}
\phi_{{h}^{'}} (w_{i}^{j}) &= 8mn+ \phi_{{g}^{'}} (w_{i}^{j}) \\
\phi_{{h}^{''}} (w_{i}^{j}) &= 4mn-1+ \phi_{{h}^{'}} (w_{i}^{j}) \\	
	\phi_{{h}^{'''}} (w_{i}^{j}) &= \phi_{{h}^{'}} (w_{i}^{j}) - 1
\end{split}
\end{equation}
The vertex sum of the vertex $w_{m+i}^{j}$ is,
\begin{equation} \nonumber
\begin{split}
\phi_{{h}^{'}} (w_{m+i}^{j}) &= 
    \begin{cases}
         4j-2+2n(i-2)+2mn &, if \ 2 \leq i \leq  2\left\lfloor \dfrac{m}{4}\right\rfloor \ \& \ i \ even \\
        4n\left\lfloor \dfrac{m}{4}\right\rfloor+4j-2+2mn &, if \  i =m, \\
				2mn+2ni+4j-2 &, if \ 2\left\lfloor \dfrac{m}{4}\right\rfloor + 2 \leq i \leq m-2 \ \& \ \ i \ even \\
				6mn-2ni+4j-2n-2 &, if \ 2 \left\lceil \dfrac{m}{4}\right\rceil + 1 \leq i \leq m-1 \ \& \ \ i \ odd \\
				6mn-4n\left\lceil \dfrac{m}{4}\right\rceil +4j-2 &, if \ i=1 for  \ m \neq 4\\
				8mn-8n\left\lceil \dfrac{m}{4}\right\rceil +2(3-i)n+4j-2 &, if \  3 \leq  i \leq 2\left\lceil \dfrac{m}{4}\right\rceil -1 \ \& \ \ i \ odd \ for \ m \neq 4\\
				2mn+8n+4j-2 &, if \  i = 1 \ for \ m = 4\\
				2mn+12n+4j-2 &, if  \ i=3 \ for \ m=4
    \end{cases}\\
\phi_{{h}^{''}} (w_{m+i}^{j}) &= 
     \begin{cases}
		 \phi_{{h}^{'}} (w_{m+i}^{j},w_{i}^{0}) + 1 +2j &, if  \ i =2 \\
		2(\phi_{{h}^{'}} (w_{m+i}^{j},w_{i}^{0})) + 1 -2mn &, if  \ i \neq 2
\end{cases}
\end{split}
\end{equation}
\text{If} $n = 2$,
\begin{equation} \nonumber
\phi_{{h}^{'''}} (w_{m+i}^{j}) =
		2({h}^{'} (w_{0}^{0},w_{m+i}^{j})) + 2mn+1
\end{equation}
\text{If} $n > 2$,
\begin{equation} \nonumber
\phi_{{h}^{'''}} (w_{m+i}^{j}) = 
  \begin{cases}
		2 ( {h}^{'} (w_{0}^{0},w_{m+i}^{j})) + 2mn+2 &, if \  i =2 \\
		2 ( {h}^{'} (w_{0}^{0},w_{m+i}^{j}) ) + 2mn+1 &, if  \ i \neq 2 
\end{cases} 
\end{equation}
The vertex sum of the vertex $w_{i}^{0}$ for each $i, 1 \leq i \leq m$ is,
\begin{equation} \nonumber
\begin{split}
\phi_{{h}^{'}} (w_{i}^{0}) &= 8mn^{2} + \phi_{{g}^{'}} (w_{i}^{0}) \\
\phi_{{h}^{''}} (w_{i}^{0}) &= \phi_{{h}^{'}} (w_{i}^{0}) - 8mn^{2} + n \\
\phi_{{h}^{'''}} (w_{i}^{0}) &= \phi_{{h}^{'}} (w_{i}^{0}) + n 
\end{split}
\end{equation}
The vertex sum of the vertex $w_{m+i}^{0}$ for each $i, 1 \leq i \leq m$ is,
\begin{equation} \nonumber
\begin{split}
\phi_{{h}^{'}} (w_{m+i}^{0}) &= 2mn^{2} + 2(\phi_{{g}^{'}} (w_{m+i}^{0}))\\
\phi_{{h}^{''}} (w_{m+i}^{0}) &= 
  \begin{cases}
		 6mn^{2} - n + \phi_{{h}^{'}} (w_{m+i}^{0}),&   if \  i \ \text{is odd} \\
		8mn^{2} - n + \phi_{{h}^{'}} (w_{m+i}^{0}),& if  \ i \ \text{is even}
\end{cases}
\end{split}
\end{equation}
\begin{equation} \nonumber
\begin{split}
\text{If} \ n = 2, \\
\phi_{{h}^{'''}} (w_{m+i}^{0}) &= 
  \begin{cases}
		 2mn^{2}+2in^{2}+n,& if \ i \ \text{is odd} \\
		6mn^{2}+2n^{2} - 2in^{2} + n,& if \ i \ \text{is even} 
\end{cases}\\
\text{If} \ n \neq 2, \\
\phi_{{h}^{'''}} (w_{m+i}^{0}) &= 
  \begin{cases}
		 2mn^{2}+2in^{2},& if \ i = 1 \\
		2mn^{2}+2in^{2}+n,& if \ i \ \text{is odd} \\
		6mn^{2}+2n^{2} - 2in^{2} + n,& if  \ i \ \text{is even} 
\end{cases}\\
\end{split}
\end{equation}
Finally, the vertex sum of the vertex $w_{0}^{j}$ for each $j, 1 \leq j \leq n$ is,
\begin{equation} \nonumber
\begin{split}
\phi_{{h}^{'}} (w_{0}^{j}) &= 8m^{2} n - 2mn +4jm -m \\
  \phi_{{h}^{''}} (w_{0}^{j}) &= \phi_{{h}^{'}} (w_{0}^{j}) - 4m^{2} n - 1 \\
  \phi_{{h}^{'''}} (w_{0}^{j}) &= 
	\begin{cases}
		 \phi_{{h}^{'}} (w_{0}^{j}) ,& if \ n = 2 \\
		\phi_{{h}^{'}} (w_{0}^{j})  - 1,& if \ n \neq 2 
\end{cases}
\end{split}
\end{equation}
The observation $(1), (2),(3)$ and $(4)$ as given in Theorem. \ref{Helmmeven} is also true under the function $h^{'}, h^{''}$ and $h^{'''}$.  Therefore, it is clear that no two vertices belongs to the same set of vertices are equal. Thus, we have the following:
\begin{itemize}
	\item The vertex sum of the vertex $w_{0}^{0}$ is greater than all the other vertex sums.
	\item If the vertex having a odd sum then we have, \\
	$\phi_{{h}^{'}} (w_{2}^{n}) < \phi_{{h}^{'}} (w_{2}^{0}), \phi_{{h}^{''}} (w_{m+3}^{n}) < \phi_{{h}^{''}} (w_{0}^{1}) < \phi_{{h}^{''}} (w_{0}^{n}) < \phi_{{h}^{''}} (w_{m+1}^{n})$.\\
		If $n = 2, \phi_{{h}^{'''}} (w_{m+3}^{n}) < \phi_{{h}^{'''}} (w_{1}^{1})$. \\
	If $n > 2, \phi_{{h}^{'''}} (w_{m+3}^{n}) < \phi_{{h}^{'''}} (w_{1}^{1}) < \phi_{{h}^{'''}} (w_{2}^{n}) < \phi_{{h}^{'''}} (w_{0}^{1})$. 
\item If the vertex having a even sum then we have, \\
$\phi_{{h}^{'}} (w_{m+3}^{n}) < \phi_{{h}^{'}} (w_{m+1}^{0}) < \phi_{{h}^{'}} (w_{m+2}^{0}) < \phi_{{h}^{'}} (w_{0}^{1}), 
\phi_{{h}^{''}} (w_{m+2}^{n}) < \phi_{{h}^{''}} (w_{1}^{1}) < \phi_{{h}^{''}} (w_{2}^{n}) < \phi_{{h}^{''}} (w_{2}^{0}).$\\
If $n = 2, \phi_{{h}^{'''}} (w_{m+2}^{0}) < \phi_{{h}^{'''}} (w_{2}^{0}) \phi_{{h}^{'''}} (w_{3}^{0}) < \phi_{{h}^{'''}} (w_{0}^{1}).  $\\
If $n > 2, \phi_{{h}^{'''}} (w_{m+2}^{n}) < \phi_{{h}^{'''}} (w_{m+2}^{0}) < \phi_{{h}^{'''}} (w_{2}^{0}).$
\end{itemize}
Hence, for any two distinct vertices we obtained distinct vertex sum. Therefore, $G$ is antimagic.\\
From lemma. (\ref{flowern=1}) and theorem. (\ref{flowermodd}), (\ref{flowermeven}), we obtain the following theorem.
\begin{theorem} \label{Flowermain}
The tensor product of Flower $Fl_m, m \geq 3$ and star $K_{1,n}, n \geq 2$  admits antimagic labeling.
\end{theorem}
\section{Conclusion}
From theorem \ref{Wheelmain}, \ref{Helmmain} and \ref{Flowermain}, we validated the result of antimagicness for the tensor product of wheel related graphs such as Wheel, Helm and Flower with Star. These results strongly supports the Hartsfield and Ringels conjecture. We also believes that, these results will motivates the researchers to study antimagic labeling for more graph products which are not done so far. Hence, we conclude the paper with the following question.\\ \\
\textbf{Problem 1.} Identify the general antimagic labeling schemes for the tensor product of $G \times H$ where $G$ and $H$ be any antimagic graph (either be connected or disconnected) which consists of odd cycle in any one of the graph $G$ or $H$?


\begin{thebibliography}{30}

\bibitem{Alon} N. Alon, G. Kaplan, A. Lev, Y. Roditty, R. Yuster, Dense graphs are antimagic, J. Graph Theory 47 (4) (2004) 297-309.
\bibitem{Baca} Martin Ba{\v{c}}a, Oudone Phanalasy,  Joe Ryan,  Andrea Semani{\v{c}}ov{\'a}-Fe{\v{n}}ov{\v{c}}{\'\i}kov{\'a},  
Antimagic labelings of join graphs, Mathematics in Computer Science, 9 (2015) 139-143.
\bibitem{ycheng1} Y. Cheng, Lattice grids and prims are antimagic, Theoretical Computer Science, 374 (1?) (2007) 66-73.
\bibitem{Ringel1} N. Hartsfield, G. Ringel, Pearls in Graph Theory, Academic Press, Inc.,
Boston, 1990, pp. 108-109; revised version, 1994.

\bibitem{Gallian} Joseph A. Gallian, A Dynamic Survey of Graph Labeling, The Electronic Journal of Combinatorics $(2022), \#DS6$.


\bibitem{LiangTCS} Y.C. Liang, X. Zhu, Anti-magic labelling of cartesian product of graphs, Theoretical Computer Science, $477 (2013) 1-5$.


\bibitem{Redel} Redek  Sl\'iva, Antimagic labeling graphs with a regular dominating subgraph, Information Processing Letters, $112, (2012), 84447$.





\bibitem{WangDM} T. M. Wang, C. C. Hsiao, On anti-magic labeling for graph products, Discrete Mathematics, 308 (16) (2008) 3624-3633.





\bibitem{Weichsel} P.M. Weichsel, The kronecker product of graphs, Proceedings of Americal Mathematical Society, 13, 1963, 47-52.

\bibitem{WenhuiAKCE} M. Wenhui, D. Guanghua, L. Yingyu, N. Wang, Lexicographic product graphs $P_m \left[P_n\right]$ are antimagic, AKCE International Journal of Graphs and Combinatorics, 15 (3) (2018) 271-283.








\bibitem{Yilma} Z. B. Yilma, Antimagic Properties of Graphs with Large Maximum Degree, Journal of Graph Theory, 72 (4) (2013) 367-373.

\bibitem{YingyuMCS} L. Yingyu, D. Guanghua, M. Wenhui, N. Wang, Antimagic labeling of the lexicographic product graph $K_{m,m} \left[P_k\right]$, Mathematics in Computer Science, 12 (1) (2018) 77-90.




















\end{thebibliography}
\end{document}